\documentclass[leqno,12pt]{amsart}
\usepackage{amssymb}
\theoremstyle{plain}
\newtheorem{Thm}{\indent\sc Theorem}[section]
\newtheorem{Prop}[Thm]{\indent\sc Proposition}
\makeatletter

\@addtoreset{equation}{section}
\makeatother 
\newcommand{\mcal}[1]{\mathcal{#1}}            
\newcommand{\msi}[1]{\text{\small $#1$}}       
\newcommand{\msii}[1]{\text{\scriptsize $#1$}} 
\newcommand{\nfrac}[2]{{\displaystyle\frac{\raisebox{-.25ex}{${}\,#1\,{}$}}{{}\,#2\,{}}}}  
\newcommand{\Z}{\mathbb{Z}}              
\newcommand{\C}{\mathbb{C}}              
\newcommand{\rank}{\mathrm{rank\,}}      
\newcommand{\PS}[1][]{\mathbb{P}_{#1}}   
\newcommand{\iu}{\mathrm{\,i\,}}         
\newcommand{\tsum}{\textstyle\sum\limits} 
\newcommand{\dsum}{\displaystyle\sum}     
\newcommand{\WC}[1]{\mcal{#1}}           
\newcommand{\PC}[1]{\mathtt{#1}}         
\newcommand{\ES}{S}                      
\newcommand{\MS}{\mcal{M}}               
\newcommand{\kt}[1]{\mathrm{#1}}         
\newcommand{\LB}[1]{\mcal{#1}}           
\newcommand{\NS}{\mathrm{NS}}            
\newcommand{\MW}{\mathrm{MW}}            
\newcommand{\Nu}{\Sigma}                 
\newcommand{\Pic}{\mathrm{Pic}}          
\newcommand{\disc}{\Delta}               
\newcommand{\RS}[1]{\mathrm{#1}}         
\newcommand{\rt}{r}                      
\newcommand{\dcl}{l}                     
\begin{document}
%
\title[On Families of Rational Elliptic Surfaces]%
      {\boldmath On Families of Rational Elliptic Surfaces \\
       with $J$-Invariant Functions of Degree One}
\author[T. Kitazawa]%
       {Takashi Kitazawa}
\address{
         College of Science and Engineering, 
         Ritsumeikan University \endgraf
         1-1-1 Nojihigashi, Kusatsu, Shiga, 525-8577, 
         Japan}
\email{takashi.kitazawa@gmail.com}
%
%
%
%
\begin{abstract}
This paper deals with a study of the rational elliptic surfaces 
whose $J$-invariant functions are of degree one. 
Almost all of these elliptic surfaces have four singular fibers, 
while the remaining surfaces have only three singular fibers. 
The moduli space of these elliptic surfaces 
is canonically isomorphic to the projective line 
by taking the $J$-values for a certain fixed type of singular fibers. 
Over the moduli space, we discuss our elliptic surfaces, 
and investigate how their sections 
are described by the parameter of the moduli space. 
By using a covering space of the moduli, 
we construct a family of our representative elliptic surfaces 
whose sections are described rationally 
by the parameter of the covering space. 
We discuss it in association 
with invariants of the regular octahedron.
\end{abstract}
\maketitle
\tableofcontents
%
%
\section*{Introduction}\label{Intro}%
%
Throughout this paper, 
we work over the complex ground field $\C$ 
(though our results might hold 
for other fields of characteristic $p \ne 2,3$). 
An elliptic surface $\ES$ refers to 
a smooth compact surface 
having a given elliptic fibration ${\varphi : \ES \to C}$ 
over a smooth compact base curve $C$. 
We assume that the fibration has a section 
${s_0 : C \to \ES}$ ${(\varphi \,\msii{\circ}\, s_0 = \mathrm{id})}$ 
which we identify with its image in $\ES$, 
and we fix it as the $0$-section of $S$. 
(Thus, we consider each regular fiber to be an elliptic curve 
endowed with the origin of the group structure, 
which is the intersection point of the fiber and the $0$-section.)
\par
Elliptic surfaces have been studied in many papers, 
including the classification of structures 
such as singular fibers and sections. 
In \cite{Ko}, Kodaira classified 
the singular fibers appearing in elliptic surfaces 
into their types with local invariants.
Miranda and Persson determined 
the possible configurations of singular fibers 
for rational elliptic surfaces in \cite{MP1},\,\cite{Pe}, 
and for semi-stable elliptic K3 surfaces in \cite{MP2}. 
For rational elliptic surfaces, 
the Mordell-Weil groups of sections 
with lattice structures 
were classified by Oguiso and Shioda in \cite{OS}. 
The list of extremal elliptic K3 surfaces 
(with Mordell-Weil groups and transcendental lattices) 
was given by Shimada and Zhang in \cite{MZ}.
\par
In order to study elliptic surfaces in detail, 
it would be advantageous to observe families of them. 
We will pick up here rational elliptic surfaces 
having $J$-invariant functions of degree one, 
and discuss them as a family of elliptic surfaces. 
This family has a most simple structure for $J$-invariant functions, 
and many other elliptic surfaces arise from it 
by using operations such as base changes and quadratic twists.
Our purpose of this paper is to describe 
several structures of this family 
(as a base for other elliptic surfaces). 
As mentioned below, the type $\mathrm{I}_0^*$ 
for Kodaira's classification of singular fibers 
plays a central role in the study of this paper. 
We investigate how our elliptic surfaces of the family are controlled 
by this type singular fibers. 
\par
For rational elliptic surfaces 
with $J$-invariant functions of degree one, 
each base curve is canonically isomorphic 
to the projective line ${\PS[1] = \PS[1](\C)}$ 
by assigning the values of the $J$-function. 
Then the base curve is considered to be the modular curve of level one
as the compactified moduli space of the elliptic curves, 
and our elliptic surfaces give elliptic fibrations 
over this modular curve.
Almost all our elliptic surfaces have four singular fibers of types 
$\mathrm{I}_1, \mathrm{II}, \mathrm{III}, \mathrm{I}_0^*$, 
and in the remaining (three) particular cases, 
they have three singular fibers of types 
$\mathrm{I}_1^*, \mathrm{II}, \mathrm{III}$, 
or $\mathrm{I}_1, \mathrm{IV}^*, \mathrm{III}$, 
or $\mathrm{I}_1, \mathrm{II}, \mathrm{III}^*$. 
The moduli space of our elliptic surfaces 
is isomorphic to the projective line $\PS[1]\!$ 
($\cong$ the modular curve of level one), 
which is canonically parametrized by taking the $J$-values 
along the singular fibers of type $\mathrm{I}_0^*$ 
($\,\mathrm{I}_1^*, \mathrm{IV}^*, \mathrm{III}^*$ in the particular cases). 
\par
In this paper, 
we will closely describe these rational elliptic surfaces 
with $J$-functions of degree one.  
In particular, we will investigate 
how their sections are described in terms of coordinates, 
depending on the moduli space of our elliptic surfaces. 
It might not be described rationally 
by the parameter of the moduli space 
(with some monodromy phenomenon in sections). 
Our main result is given by Theorem \ref{MainThm}.
By using a octahedral covering space of the moduli, 
we construct a family of our elliptic surfaces 
for which their sections are described rationally 
by the parameter of the covering space. 
This covering space of the moduli is branched 
at the three points of $J$-values $\infty, 0,1$, 
which correspond to the above mentioned 
three particular cases of our elliptic surfaces, 
and they are associated with 
the edges, faces, vertices of the regular octahedron.  
\par
Under the above result, we investigate 
how the sections of our elliptic surfaces pass through a fixed fiber. 
In Theorem \ref{Thm2}, we describe it for regular fibers 
in terms of Weierstrass $\wp$-functions of complex tori. 
We mention it in association with 
the $2$-division points of the elliptic curves 
determined by the singular fibers of type $\kt{I}_0^*$.
In Theorem \ref{Thm1}, 
for singular fibers of types $\kt{I}_1$ and $\kt{II}$, 
we discuss the intersection points passed through by the sections. 
We describe them in the complex numbers 
by using certain polynomials of degree $4$ 
defined on  our octahedral covering space of the moduli. 
\par 
In Section \ref{Sec:1}, 
we recall several facts as preliminaries 
concerning rational elliptic surfaces. 
In Section \ref{Sec:2}, 
we construct rational elliptic surfaces 
with $J$-functions of degree one 
by using Weierstrass models 
together with birational correspondence to pencils of cubic curves 
on the projective plane ${\PS[2] = \PS[2](\C)}$.
In Section \ref{Sec:3}, 
we describe sections for our rational elliptic surfaces 
in the sense of divisor classes,
which is associated with the blowing-up of $\PS[2]$ 
as the resolution of the nine base points for our cubic pencil. 
In Section \ref{Sec:4}, 
we present explicit coordinate representation 
for sections of our elliptic surfaces 
by using the octahedral covering space of the moduli.
In Section \ref{Sec:5}, 
we describe intersection points 
which are passed through by sections in a fixed fiber.
In Section \ref{Sec:6}, 
we give a characterization of our octahedral covering space 
with a relationship to the modular curve 
of the principal congruence subgroup of level $4$, 
which we describe through a different point of view 
from the type $\kt{I}_0^*$ of singular fibers. 
In Appendix, we mention the group action of covering transformations 
on the polynomials used in this paper.
\section{Rational elliptic surfaces}\label{Sec:1}%
We will recall several facts concerning rational elliptic surfaces 
(see \cite{Mi}, \cite{SS} for details). 
Let $\ES$ be a (relatively minimal) rational elliptic surface 
with given $0$-section $s_0$,  
(where the minimality means that no $(-1)$-curve 
is contained in any fiber of the fibration ${\varphi : \ES \to C\,}$). 
The elliptic surface $\ES$ (over the ground field $\C\,$) 
is described as the (minimal) desingularization 
of the Weierstrass model defined by the form 
\begin{equation}\label{Weierstrass}
{}\hspace{8ex}
\WC{Y}^2\WC{Z} 
= 4\WC{X}^3 - g_2 \WC{XZ}^2 -g_3 \WC{Z}^3
\ \ \ \ \ (\text{\,with $s_0$ as $\WC{X}=\WC{Z}=0\,$})
\end{equation}
in the $\PS[2]$-bundle 
${\PS(\LB{L}^2 \oplus \LB{L}^3 \oplus \LB{O}_C)}$ 
over the base curve $C$ (cf.\ \cite{Ka}). 
In this model, 
$\LB{L}$ is the line bundle 
determined by the dual of the higher direct image $R^1 \varphi_* (\LB{O}_S)$, 
and $g_2, g_3$ are global sections of 
$\LB{L}^4, \LB{L}^6 \,
 \big( g_2 \,{\in}\, H^0 (C, \LB{L}^4),\, 
       g_3 \,{\in}\, H^0 (C, \LB{L}^6) \big)$ 
with the following conditions:
\begin{enumerate}
\item The discriminant $\disc := g_2^{\,3} - 27 g_3^{\,2} 
      \,\big( {\in}\, H^0 (C, \LB{L}^{12}) \big)$ 
      is not identically zero.
\item The vanishing orders $v_p(g_2), v_p(g_3)$ of $g_2, g_3$ 
      satisfy $v_p(g_2) < 4$ or $v_p(g_3) < 6$\\
      at any point $p$ on $C$.
\end{enumerate}
By the rationality of $\ES$, 
we have ${C \cong \PS[1]}$ and ${\LB{L} \cong \LB{O}_{\PS[1]}(1)}$, 
and then the Weierstrass coefficients $g_2, g_3$ are considered to be 
binary homogeneous polynomials of degree $4, 6$ 
with respect to homogeneous coordinates of $C$ 
(which we denote by ${(\PC{S} : \PC{T})}$). 
The Weierstrass model (\ref{Weierstrass}) 
is locally represented in the direct product ${\C \times \PS[2]}$, 
that is, represented by the following local form over the patch of $C$ 
with local coordinate $\frac{\,\PC{T}\,}{\,\PC{S}\,}$:
\begin{equation*}
Y_1^{\,2} Z_1^{}
 = 4 X_1^{\,3} - \nfrac{{g\mathstrut}_2}{\PC{S}^4\!} X_1^{} Z_1^{\,2} 
               - \nfrac{{g\mathstrut}_3}{\PC{S}^6\!} Z_1^{\,3}
\ \ \text{with}\ \ 
(X_1 : Y_1 : Z_1) 
= \Big( \nfrac{\WC{X}}{\,\PC{S}^2\!} 
        : \nfrac{\WC{Y}}{\,\PC{S}^3\!} 
        : \nfrac{\WC{Z}}{1\mathstrut} \Big),
\end{equation*}
and represented by the following local form 
over the patch with local coordinate $\frac{\,\PC{S}\,}{\,\PC{T}\,}$:
\begin{equation*}
Y_2^{\,2} Z_2^{}
 = 4 X_2^{\,3} - \nfrac{{g\mathstrut}_2}{\PC{T}^4\!} X_2^{} Z_2^{\,2} 
               - \nfrac{{g\mathstrut}_3}{\PC{T}^6\!} Z_2^{\,3}
\ \ \text{with}\ \ 
(X_2 : Y_2 : Z_2) 
= \Big( \nfrac{\WC{X}}{\,\PC{T}^2\!} 
        : \nfrac{\WC{Y}}{\,\PC{T}^3\!} 
        : \nfrac{\WC{Z}}{1\mathstrut} \Big).
\end{equation*}
They are patched up 
by the following relations of the locally trivialized coordinates:
\begin{equation*}
\nfrac{\PC{S}}{\PC{T}} = \Big( \nfrac{\PC{T}}{\PC{S}} \Big)^{\!\!-1} 
\ \ \text{and} \ \ 
(X_2 : Y_2 : Z_2 ) = 
\Big( \Big( \nfrac{\PC{T}}{\PC{S}} \Big)^{\!\!-2} X_1  
 : \Big( \nfrac{\PC{T}}{\PC{S}} \Big)^{\!\!-3} Y_1  
 : Z_1 \Big).
\end{equation*}
\par
For the Weierstrass model (\ref{Weierstrass}) 
of the elliptic surface $\ES$, 
we can read off the fiber types of Kodaira's classification 
appearing over the points on the base curve $C$. 
In fact, the fiber type over a point ${p \in C}$ 
is determined by the vanishing orders 
$v_p(g_2), v_p(g_3), v_p(\disc)$ 
as shown in Table \ref{Kodaira} (cf.\,\cite{MP1}). 
In this table, the type $\kt{I}_0$ 
means a regular fiber (elliptic curve), 
and the other types mean singular fibers. 
In detail, 
we consider the type $\kt{I}_1$ for a nodal rational curve, 
the type $\kt{II}$ for a cuspidal rational curve, 
the type $\kt{III}$ for the collection of two tangential rational curves, 
the type $\kt{IV}$ for the collection of three concurrent rational curves, 
and the types ${\kt{I}_n \,(\,n \ge 2\,)}$, ${\kt{I}_n^* \,(\,n \ge 0\,)}$, 
$\kt{IV}^*$, $\kt{III}^*$, $\kt{II}^*$ 
for the collections of several rational curves with intersection forms
of the extended diagrams of types 
$\widetilde{\RS{A}}_{n-1}$, $\widetilde{\RS{D}}_{n+4}$, $\widetilde{\RS{E}}_6$, 
$\widetilde{\RS{E}}_7$, $\widetilde{\RS{E}}_8$, respectively 
(in the usual sense). 
\begin{table}[htb]
\centering\normalsize
\caption{Kodaira's fiber types}
\label{Kodaira}
$\begin{array}{@{\ }c@{\ } *{5}{| @{\ }c@{\ }} }
\text{type} & v_p(g_2) & v_p(g_3) & v_p(\disc) & J(p) & e_p(J)\\
\hline
\hline
         & 0     & 0     & 0 & \ne 0,1,\infty & \ge 1 \\
\kt{I}_0 & \ge 1 & 0     & 0 & 0              & 3v_p(g_2) \\
         & 0     & \ge 1 & 0 & 1              & 2v_p(g_3) \\
\hline
\kt{I}_n (n \ge 1) & 0  & 0  & n & \infty & n \\
\hline
           & 2     & 3     & 6 & \ne 0,1,\infty & \ge 1 \\
\kt{I}_0^* & \ge 3 & 3     & 6 & 0              & 3v_p(g_2)-6 \\
           & 2     & \ge 4 & 6 & 1              & 2v_p(g_3)-6 \\
\hline
\kt{I}_n^* (n \ge 1) & 2  & 3  & n+6 & \infty & n \\
\hline
\kt{II}  & \ge 1 & 1     & 2  & 0  & 3v_p(g_2)-2 \\
\hline
\kt{III} & 1     & \ge 2 & 3  & 1  & 2v_p(g_3)-3 \\
\hline
\kt{IV}  & \ge 2 & 2     & 4  & 0  & 3v_p(g_2)-4 \\
\hline
{}\ \kt{IV}^*  & \ge 3 & 4     & 8  & 0  & 3v_p(g_2)-8 \\
\hline
{}\ \kt{III}^* & 3     & \ge 5 & 9  & 1  & 2v_p(g_3)-9 \\
\hline
{}\ \kt{II}^*  & \ge 4 & 5     & 10 & 0  & 3v_p(g_2)-10 \\
\hline
\end{array}$
\end{table}
\par
In Table \ref{Kodaira}, 
we have additional information concerning $J(p)$ and $e_p(J)$. 
These refer to the value and the ramification index at ${p \in C}$ 
for the $J$-invariant of $\ES$, 
where the $J$-invariant is the rational function 
(as the regular mapping from $C$ to $\PS[1]$)
determined by the values of $J$-modulus 
for the elliptic curves as the regular fibers. 
On the Weierstrass model (\ref{Weierstrass}), 
the $J$-invariant is represented by the following form:
\begin{equation}
{}\qquad\qquad
J = \nfrac{{g\mathstrut}_2{}^{\!\!3}}
          {\disc} 
  = \nfrac{{g\mathstrut}_2{}^{\!\!3}}
          {{g\mathstrut}_2{}^{\!\!3} - 27 {g\mathstrut}_3{}^{\!\!2}} 
\qquad ( \text{\,with}\ j = 1728\,J\,).
\end{equation}
Since the elliptic surface $\ES$ is rational, 
the $J$-invariant function ${J : C \to \PS[1]}$ 
is at most of degree $12$. 
In particular, it is of degree $1$ in the study of this paper.
\par
Next, we recall another description 
of rational elliptic surfaces. 
We can describe the surface $\ES$ 
as a blowing-up of $\PS[2]$ (successively in nine points). 
The center of the blowing-up consists 
of (possibly infinitely near) nine base points 
determined by a pencil of plane cubic curves.
In this description, 
the fibers of $S$ are described from the cubic curves in the pencil, 
and the $0$-section $s_0$ is given 
by the last ($9$-th) exceptional curve of the blowing-up.
\par
The N\'eron-Severi group (N\'eron-Severi lattice)
$\NS(\ES) \,\big( {\cong}\ \Pic(\ES)/\Pic^0(\ES) \big)$ 
is isometric to the unimodular odd lattice 
of rank $10$ with signature $(1,9)$. 
It is generated by the divisor classes $\dcl$ and $e_1, \cdots{}, e_9$ 
associated with a line in $\PS[2]$ 
and the nine exceptional curves of the blowing-up 
(more precisely, the total transforms of them).
We set here $r_{ij}, r_{ijk}$ and $f$ to be the following classes:
\begin{equation*}
r_{ij} = e_i - e_j,\ \ \ 
r_{ijk} = \dcl - e_i - e_j - e_k, \ \ \ 
f = 3\dcl - \tsum_{k=1}^9 e_k.
\end{equation*}
The class of the fibers 
for the elliptic fibration ${\varphi : \ES \to C}$ 
is given by $f$, which is the anti-canonical class of $\ES$. 
The fiber class $f$ and the $0$-section $s_0$ 
generate a unimodular sublattice of $\NS(\ES)$ with signature $(1,1)$, 
which we denote by ${U = \langle\, f, s_0 \,\rangle}$. 
Then the orthogonal complement $U^{\bot}$ is isometric 
to the negative definite root lattice of type $\RS{E}_8$. 
Its fundamental system of roots is given by 
$r_{12}, \cdots{}, r_{78}$ and $r_{123}$ with the diagram below.
\begin{equation}\label{EE6}
 \begin{array}{c *{12}{@{\ }c}}
  r_{12} & \rule[.5mm]{1em}{.8pt} &
  r_{23} & \rule[.5mm]{1em}{.8pt} &
  r_{34} & \rule[.5mm]{1em}{.8pt} &
  r_{45} & \rule[.5mm]{1em}{.8pt} &
  r_{56} & \rule[.5mm]{1em}{.8pt} & 
  r_{67} & \rule[.5mm]{1em}{.8pt} &
  r_{78} \ (\,\rule[.5mm]{1em}{.8pt} \ r_{89} )\\
   & & & & \rule[-1mm]{.8pt}{1em} & & & & & & & & \\
   & & & & \rt_{123} & & & & & & & & \\
 \end{array}
\end{equation}
The above extended diagram of type $\widetilde{\RS{E}}_8$, 
to which $r_{89}$ is added, 
generates the direct sum 
$\langle \,f\, \rangle \oplus (U^{\bot})$ 
(which coincides with $\langle \,f\, \rangle^{\bot}$). 
\par
The orthogonal projection 
from ${\NS(\ES) = U \oplus (U^\bot})$ onto $U^\bot$ 
induces a one-to-one correspondence onto $U^\bot$ 
by restricting the projection within the subset 
$\Nu(\ES) \,\big( {\subset}\,\NS(\ES) \big)$ 
of numerical sections of $\ES$, 
where a numerical section $s$ is 
an effective divisor (class in $\NS(\ES)$) 
with intersection numbers $(s \cdot f)=1$ and $(s \cdot s)=-1$ 
(cf.\ \cite{MoP}). 
For an element $r$ in $U^\bot$, 
the corresponding numerical section $s$ is given by the following:
\begin{equation*}
{}\qquad
s = s_0 - \msi{\nfrac{1}{2}} (r \cdot r) f + r \qquad
\Big({\,\Leftrightarrow\,}\ 
r = s - s_0 + \msi{\nfrac{1}{2}} \big( (s-s_0) \!\cdot\! (s-s_0) \big) f 
\ \Big),
\end{equation*}
in which ${(r \cdot r) \,\big({=}\ (s-s_0) \!\cdot\! (s-s_0) \big)}$ 
is an even integer since the root lattice $U^\bot$ is an even lattice. 
\par
Although $\Nu(\ES)$ is not a subgroup of $\NS(\ES)$, 
there is a group structure induced from $U^\bot$ onto $\Nu(\ES)$ 
under the one-to-one correspondence.
For numerical sections $s_1, s_2$ in $\Nu(\ES)$ 
(corresponding to $r_1, r_2$ in $U^{\bot}$), 
we have $s_1 \,\dot{+}\, s_2$ as follows:
\begin{equation*}
\begin{array}{l@{\ }c@{\ }l}
s_1 \,\dot{+}\, s_2 
 &=& s_0 - \msi{\nfrac{1}{2}} \big( (r_1+r_2) \!\cdot\! (r_1+r_2) \big) f 
     + r_1 + r_2 \\[1ex]
 &=& s_0 + (s_1-s_0) + (s_2-s_0) 
     - \big( (s_1-s_0) \!\cdot\! (s_2-s_0) \big) f.
\end{array}
\end{equation*}
On each regular fiber, this addition is 
just the group law of the elliptic curve.
\par
Numerical sections might not be irreducible, 
which could contain fiber components of the elliptic fibration.
The factor group of $\Nu(\ES)$ 
by the equivalence relation modulo the fiber components 
is isomorphic to the group of irreducible sections 
(taking $s_0$ as the zero), 
which is the Mordell-Weil group $\MW(\ES)$ 
of the elliptic surface $S$. 
Under the one-to-one correspondence between $\Nu(\ES)$ and $U^\bot$, 
the Mordell-Weil group is isomorphic to a factor group 
by a sublattice $L$ in the root lattice $U^\bot$ of type $\RS{E}_8$. 
This sublattice $L$ is generated 
by the fiber components not meeting the $0$-section $s_0$.
We then have the following isomorphisms: 
\begin{equation}
\MW(\ES) \,\cong\, U^\bot / L \,\cong\, \NS(\ES) / (U \oplus L).
\end{equation}
Moreover, according to Shioda's work (cf.\ \cite{Shi}), 
the Mordell-Weil group $\MW(\ES)$ (up to torsion) 
has a lattice structure (generally not integral), 
which is the (sign changed) dual lattice 
of the orthogonal complement $(U \oplus L)^{\bot}$. 
By the above isomorphisms, 
we have the following relation (usually called the Shioda-Tate formula):
\begin{equation}\label{Shioda-Tate}
\rank \NS(\ES) = 2 + \rank L + \rank \MW(\ES).
\end{equation}
In particular, 
we have $\rank \MW(\ES) = 8 - \rank L$ 
for the rational elliptic surface $\ES$ 
(since we have $\rank \NS(\ES) = 10$).
\section{Elliptic fibrations of degree one}\label{Sec:2}%
In this section, 
we will describe rational elliptic surfaces 
with $J$-functions of degree one. 
Let $\ES$ be such a rational elliptic surface. 
Then the $J$-function gives 
an isomorphism from the base curve $C$ onto $\PS[1]$ 
together with the ramification index ${e_p(J) = 1}$ for any point $p \in C$.
Thus, the appearing fiber type over the point $p$ 
is restricted within the two types shown in Table \ref{FiberType} 
depending on the $J$-value $J(p)$. 
\begin{table}[htb]
\centering\normalsize
\caption{Fiber types for elliptic fibrations of degree one}%
\label{FiberType}%
$\begin{array}{@{\ }c@{\ } *{4}{| @{\ }c@{\ }} }
\text{type} & v_p(g_2) & v_p(g_3) & v_p(\disc) & J(p) \\
\hline
\hline
\kt{I}_0 & 0  & 0  & 0  & \ne 0,1,\infty \\
\kt{I}_0^* & 2  & 3  & 6  & \ne 0,1,\infty \\
\hline
\kt{I}_1   & 0  & 0  & 1 & \infty \\
\kt{I}_1^* & 2  & 3  & 7 & \infty \\
\hline
\end{array} \ \ \ 
\begin{array}{@{\ }c@{\ } *{4}{| @{\ }c@{\ }} }
\text{type} & v_p(g_2) & v_p(g_3) & v_p(\disc) & J(p) \\
\hline
\hline
\kt{II}   & 1  & 1  & 2  & 0 \\
{}\ \kt{IV}^* & 3  & 4  & 8  & 0 \\
\hline
\kt{III}   & 1  & 2  & 3 & 1 \\
{}\ \kt{III}^* & 3  & 5  & 9 & 1 \\
\hline
\end{array}$%
\end{table}
\par
The Euler number of $\ES$ is equal to $12$ 
(for the rationality of $\ES$), 
and it coincides with the sum of Euler numbers 
for the singular fibers of $\ES$, 
which is the same to the sum of vanishing orders $v_p(\disc)$. 
Therefore, by combinatorial discussion, 
the configurations of singular fibers 
are restricted within the following four cases: 
\begin{equation*}
[\, \kt{I}_1, \kt{II}, \kt{III}, \kt{I}_0^*\,], \ \ 
[\, \kt{I}_1^*, \kt{II}, \kt{III} \,], \ \ 
[\, \kt{I}_1, \kt{IV}^*, \kt{III} \,], \ \ 
[\, \kt{I}_1, \kt{II}, \kt{III}^* ].
\end{equation*}
(In Persson's list of \cite{Pe}, 
these configurations are realized, 
discussing double covers of $\PS[2]$ 
branched along plane quartic curves.) 
%
%
\par
The first case $[\, \kt{I}_1, \kt{II}, \kt{III}, \kt{I}_0^* \,]$ 
means that $\ES$ has four singular fibers of types 
$\kt{I}_1, \kt{II}, \kt{III}$ and $\kt{I}_0^*$. 
The $J$-values of the fibers 
$\kt{I}_1, \kt{II}, \kt{III}$ are given by $\infty, 0, 1$, 
and the $J$-value $J_0$ of the fiber $\kt{I}_0^*$ 
is given by some complex number ${(\ne \infty, 0, 1)}$. 
As the cases for ${J_0 = \infty, 0, 1}$, 
we have the remaining three cases 
$[\, \kt{I}_1^*, \kt{II}, \kt{III} \,]$, 
$[\, \kt{I}_1, \kt{IV}^*, \kt{III} \,]$, 
$[\, \kt{I}_1, \kt{II}, \kt{III}^* ]$. 
They are considered to be the cases that 
the singular fiber $\kt{I}_0^*$ in the general case 
$[\, \kt{I}_1, \kt{II}, \kt{III}, \kt{I}_0^* \,]$ 
becomes confluent with one of the other singular fibers 
$\kt{I}_1, \kt{II}, \kt{III}$ in the following sense: 
\begin{equation*}
\kt{I}_0^* + \kt{I}_1 \to \kt{I}_1^*, \ \ \ 
\kt{I}_0^* + \kt{II} \to \kt{IV}^*, \ \ \ 
\kt{I}_0^* + \kt{III} \to \kt{III}^*,
\end{equation*}
which are consistent with respect to 
the vanishing orders $v_p(g_2), v_p(g_3), v_p(\disc)$ 
mentioned in Table \ref{FiberType}.
We note that the above confluences are in restricted cases 
by the condition of degree-one $J$-function, 
and thus isotrivial cases (having constant $J$-functions) 
are not taken into consideration in this paper. 
\par
We now construct $\ES$ by using the Weierstrass model. 
Since the base curve $C$ is considered to be $\PS[1]$ 
and the $J$-function is of degree one, 
we give homogeneous coordinates ${(\PC{S}:\PC{T})}$ to $C$ 
such that ${J = \frac{\,\PC{T}\,}{\,\PC{S}\,}}$. 
The Weierstrass coefficients $g_2, g_3$ are represented
by homogeneous polynomials of degree $4, 6$ 
with respect to ${(\PC{S} : \PC{T})}$, 
and then we have the following:
%
%
\begin{Prop}
For a rational elliptic surface $\ES$ with $J$-function of degree one, 
its Weierstrass model is given by the following form 
$($on the canonical homogeneous coordinates ${(\PC{S} : \PC{T})}$ 
of the base curve $C$ with $J$-function 
${J = \frac{\,\PC{T}\,}{\,\PC{S}\,}}){:}$
\begin{equation}\label{WModel}
\WC{Y}^2\WC{Z} 
= 4\WC{X}^3 - g_2 \WC{XZ}^2 -g_3 \WC{Z}^3 \ \ \ 
\Bigg(
\begin{array}{@{\,}l@{}}
g_2 = 27\,\PC{T}\,(\PC{T} - \PC{S}) (a\,\PC{T} - b\,\PC{S} )^2 \\[1ex]
g_3 = 27\,\PC{T}\,(\PC{T} - \PC{S})^2 (a\,\PC{T} - b\,\PC{S} )^3
\end{array}
\Bigg),
\end{equation}
where ${(a,b) \,\big({\ne}\,(0,0) \big)}$ is considered 
to be a homogeneous representative 
of a point on $C\,({\cong}\,\PS[1])$. 
This point $($equivalently 
the inhomogeneous value ${J_0 = \frac{\,b\,}{a}})$ 
determines that the elliptic surface $\ES$ is in the case 
$[\, \kt{I}_1^*, \kt{II}, \kt{III} \,]$, 
$[\, \kt{I}_1, \kt{IV}^*, \kt{III} \,]$, 
$[\, \kt{I}_1, \kt{II}, \kt{III}^* ]$, or 
$[\, \kt{I}_1, \kt{II}, \kt{III}, \kt{I}_0^*\,]$ 
for ${J_0 = \infty, 0, 1,}$ or the other values, respectively.
\end{Prop}
\begin{proof}
In the case $[\, \kt{I}_1, \kt{II}, \kt{III}, \kt{I}_0^*\,]$ 
with the $J$-value ${J_0\,(\ne \infty, 0, 1)}$ of the fiber $\kt{I}_0^*$, 
the singular fibers $\kt{I}_1, \kt{II}, \kt{III}, \kt{I}_0^*$ 
appear at the points ${\frac{\,\PC{T}\,}{\,\PC{S}\,} = \infty, 0, 1, J_0}$ 
(since we have ${J = \frac{\,\PC{T}\,}{\,\PC{S}\,}}$ 
for the homogeneous coordinates ${(\PC{S}:\PC{T})}$ of $C\,$). 
By the vanishing orders mentioned above in Table \ref{FiberType}, 
the Weierstrass coefficients $g_2, g_3$ 
should be written as follows:
\begin{equation*}
g_2 = k_2\,\PC{T}\,(\PC{T} - \PC{S}) (\PC{T} - J_0\,\PC{S})^2, \ \ 
g_3 = k_3\,\PC{T}\,(\PC{T} - \PC{S})^2 (\PC{T} - J_0\,\PC{S})^3 \ \ \ 
(\,0 \ne k_2, k_3 \in \C\,),
\end{equation*}
and then the discriminant ${\disc = g_2^{\,3} - 27 g_3^{\,2}}$ 
is written by the following form:
\begin{equation*}
\disc = \PC{T}^2 (\PC{T} - \PC{S})^3 (\PC{T} - J_0\,\PC{S})^6 
        \{ (k_2^{\,3} - 27 k_3^{\,2})\,\PC{T} + 27 k_3^{\,2} \,\PC{S} \},
\end{equation*}
which must vanish at ${\frac{\,\PC{T}\,}{\,\PC{S}\,} = \infty}$ 
for the fiber $\kt{I}_1$. 
Thus, we have ${k_2^{\,3} = 27 k_3^{\,2}}$. 
By setting ${k_2 = 27 a^3}$, ${k_3 = 27 a^2}$ and ${b = a J_0}$, 
we obtain the Weierstrass model (\ref{WModel}). 
Similarly, we can determine $g_2, g_3$ 
in the remaining cases 
$[\, \kt{I}_1^*, \kt{II}, \kt{III} \,]$, 
$[\, \kt{I}_1, \kt{IV}^*, \kt{III} \,]$, 
$[\, \kt{I}_1, \kt{II}, \kt{III}^* ]$, 
which are obtained by ${\frac{b}{\,a\,} = \infty, 0, 1}$, respectively.
\end{proof}
%
\par
For the Weierstrass model (\ref{WModel}), if we exchange ${(a,b)}$ 
for another homogeneous representative ${(ka, kb)} \ {(k \in \C^*)}$, 
then $g_2, g_3$ are transformed to $k^2g_2, k^3g_3$, 
and they determine the same elliptic surface $\ES$ 
under the following twisting transformation (by ${\!\sqrt{k} \in \C\,}$) 
of the Weierstrass model:
\begin{equation}\label{twist}
(\,\WC{X} : \WC{Y} : \WC{Z}\,) \ \mapsto\  
( \sqrt{k}{}^{\,2} \WC{X} : \sqrt{k}{}^{\,3} \WC{Y} : \WC{Z}\,)
= (\, k \WC{X} : k \sqrt{k}\,\WC{Y} : \WC{Z}\,).
\end{equation}
Thus, the isomorphism classes of our rational elliptic surfaces 
(with $J$-functions of degree one) 
form a one-dimensional variety $\MS$
as the moduli space. 
This moduli space $\MS$ is realized to be the projective line $\PS[1]$ 
with homogeneous coordinates ${(a:b)}$, 
which is canonically parametrized by the $J$-value 
${J_0 = \frac{b}{\,a\,}}$ for the singular fiber $\kt{I}_0^*$ 
(${J_0 = \infty, 0, 1}$ for fibers 
$\kt{I}_1^*, \kt{IV}^*, \kt{III}^*$ in the particular cases). 
\par
It seems impossible to construct 
a three-dimensional total space 
as an analytic family representing our rational elliptic surfaces 
over the whole moduli space $\MS$. 
In fact, over the patch of $\MS$ 
with local coordinate $\frac{\,b\,}{a}$, 
we have the following local family 
of representative Weierstrass equations:
\begin{equation*}
\WC{Y}_1^{\,2}\WC{Z}_1^{} 
 = 4\WC{X}_1^{\,3} 
   - g_2^{_{(1)}} \WC{X}_1^{} \WC{Z}_1^{\,2} 
   - g_3^{_{(1)}} \WC{Z}_1^{\,3} 
\ \ \ 
\Big(\, g_2^{_{(1)}} = \nfrac{{g\mathstrut}_2}{a^2\!},\ 
        g_3^{_{(1)}} = \nfrac{{g\mathstrut}_3}{a^3\!} \,\Big),
\end{equation*}
and, over the patch with local coordinate $\frac{\,a\,}{b}$, 
we have the following local family:
\begin{equation*}
\WC{Y}_2^{\,2}\WC{Z}_2^{} 
 = 4\WC{X}_2^{\,3} 
   - g_2^{_{(2)}} \WC{X}_2^{} \WC{Z}_2^{\,2} 
   - g_3^{_{(2)}} \WC{Z}_2^{\,3} 
\ \ \ 
\Big(\, g_2^{_{(2)}} = \nfrac{{g\mathstrut}_2}{b^2\!},\ 
        g_3^{_{(2)}} = \nfrac{{g\mathstrut}_3}{b^3\!} \,\Big).
\end{equation*}
Then we have the following relations 
among the local coordinates of $\MS$ and 
the local versions of 
the Weierstrass coefficients $g_2, g_3$ of (\ref{WModel}):
\begin{equation*}
\nfrac{a}{b} = \Big( \nfrac{b}{a} \Big)^{\!\!-1}, \ \ \ 
g_2^{_{(2)}} = \Big( \nfrac{b}{a} \Big)^{\!\!-2} g_2^{_{(1)}}, \ \ \   
g_3^{_{(2)}} = \Big( \nfrac{b}{a} \Big)^{\!\!-3} g_3^{_{(1)}}.
\end{equation*}
However, $g_2, g_3$ are homogeneous forms of degree $2, 3$ 
with respect to the parameter ${(a : b)}$ of $\MS$. 
Thus, we cannot suitably patch up the Weierstrass coordinates 
${(\WC{X}_1 : \WC{Y}_1 : \WC{Z}_1)}$ 
and ${(\WC{X}_2 : \WC{Y}_2 : \WC{Z}_2)}$ 
because of the ramification of the square root 
appearing in the twisting transformation (\ref{twist}).
For global construction of the family, 
it would be necessary to change the parameter space $\MS$ 
for its covering space of even degree 
(in order to make the degrees of $g_2, g_3$ divisible by $4, 6$, respectively). 
In Section 4, we will use a octahedral covering space of degree $24$ 
in order to describe sections for our elliptic surfaces explicitly.
\par
Next, we mention another description 
of the rational elliptic surface $\ES$ 
as a pencil of plane cubic curves, 
which corresponds to the Weierstrass model (\ref{WModel}). 
%
\begin{Prop}
The rational elliptic surface $\ES$ 
of the Weierstrass model $(\ref{WModel})$ 
is obtained from the following pencil of plane cubic curves\/$:$
\begin{equation}\label{Cubics}
\PC{T}\,(\,a\,Y^2Z - 4X^3 + 27XZ^2 + 27Z^3) = \PC{S}\,(\,b\,Y^2Z - 4X^3),
\end{equation}
where ${(X:Y:Z)}$ is a homogeneous coordinate system of\/ $\PS[2]$, 
and ${(\PC{S} : \PC{T})}$ gives a parameter of the pencil.
\end{Prop}
\begin{proof}
We construct a birational correspondence between the coordinates 
${(X:Y:Z)}$ of $\PS[2]$ and 
${(\WC{X} : \WC{Y} : \WC{Z})}$ of the Weierstrass model. 
We determine it by the following relations:
\begin{equation}\label{Transform}
\WC{X} = (\PC{T} - \PC{S})(a\,\PC{T} - b\,\PC{S})\,X, \ \ 
\WC{Y} = (\PC{T} - \PC{S})(a\,\PC{T} - b\,\PC{S})^2\,Y, \ \
\WC{Z} = Z.
\end{equation}
Then the Weierstrass model (\ref{WModel}) 
is birationally transformed to the following form:
\begin{equation}\label{CubicsWform}
(a\,\PC{T} - b\,\PC{S})\,Y^2Z = 4(\PC{T} - \PC{S})\, X^3 -27\,\PC{T}\, XZ^2 - 27\,\PC{T}\, Z^3,
\end{equation}
and it is the same to the cubic pencil (\ref{Cubics}). 
In this transformation, we have removed the common factor 
$(\PC{T} - \PC{S})(a\,\PC{T} - b\,\PC{S})^2$, 
which is not zero for almost all fibers.
\end{proof}
%
\par
In the general case of ${\frac{b}{\,a\,} \ne \infty, 0, 1}$, 
we have four singular cubic curves 
in the members of the cubic pencil (\ref{Cubics}), 
which appear at ${\frac{\,\PC{T}\,}{\,\PC{S}\,} = \infty, 0, 1, \frac{\,b\,}{\,a\,}}$ 
as follows:
\begin{equation}\label{Members}
\left\{
\begin{array}{@{\ }l@{\quad\ }l@{\,}}
a\,Y^2Z = (X-3Z)(2X+3Z)^2 
    & \text{(nodal cubic curve),} \\[1ex]
b\,Y^2Z = 4X^3  
    & \text{(cuspidal cubic curve),} \\[1ex]
\{ (b-a)Y^2 - 27XZ - 27 Z^2 \} Z = 0  
    & \text{(conic and tangent line),} \\[1ex]
4(b-a)X^3 - 27b\,XZ^2 - 27b\,Z^3 = 0 
    & \text{(three concurrent lines).}
\end{array}
\right.
\end{equation}
The cubic pencil has nine base points. 
Six of them are particularly 
on the above conic of the third singular cubic curve, 
lying on the last singular cubic curve of the three concurrent lines.
The remaining three base points are infinitely near 
to the inflection point ${(0:1:0)}$ of all smooth cubic curves 
in the pencil (having the common tangent ${Z=0}$ at the inflection point). 
This inflection point is the intersection point of the three concurrent lines.
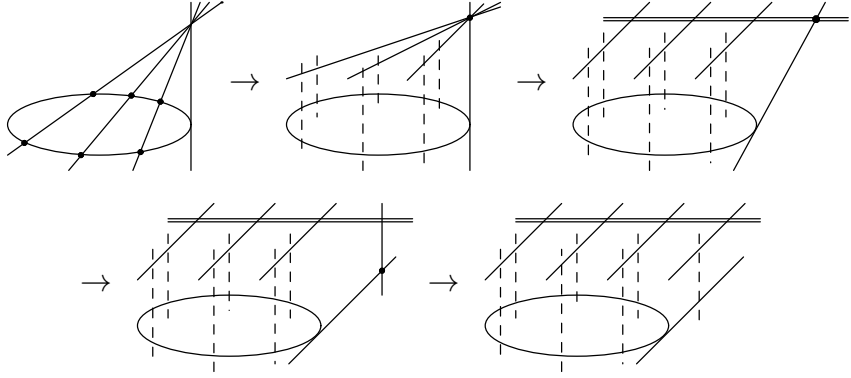
\begin{figure}[htb]
\centering\normalsize
\unitlength 0.1in
\begin{picture}( 12.4000, 10.4000)(  1.0000,-11.4000)
%
\special{pn 8}%
\special{ar 660 820 480 160  0.0000000 6.2831853}%
%
\special{pn 8}%
\special{pa 1140 180}%
\special{pa 1140 1060}%
\special{fp}%
%
\special{pn 13}%
\special{pa 1300 180}%
\special{pa 1300 180}%
\special{fp}%
%
\special{pn 8}%
\special{pa 1300 180}%
\special{pa 180 980}%
\special{fp}%
%
\special{pn 8}%
\special{pa 500 1060}%
\special{pa 1236 180}%
\special{fp}%
%
\special{pn 8}%
\special{pa 836 1060}%
\special{pa 1188 180}%
\special{fp}%
%
\special{pn 8}%
\special{pa 100 100}%
\special{pa 1340 100}%
\special{pa 1340 1140}%
\special{pa 100 1140}%
\special{pa 100 100}%
\special{ip}%
%
\special{pn 13}%
\special{sh 1}%
\special{ar 268 916 10 10 0  6.28318530717959E+0000}%
\special{sh 1}%
\special{ar 268 916 10 10 0  6.28318530717959E+0000}%
%
\special{pn 13}%
\special{sh 1}%
\special{ar 564 980 10 10 0  6.28318530717959E+0000}%
\special{sh 1}%
\special{ar 564 980 10 10 0  6.28318530717959E+0000}%
%
\special{pn 13}%
\special{sh 1}%
\special{ar 876 964 10 10 0  6.28318530717959E+0000}%
\special{sh 1}%
\special{ar 876 964 10 10 0  6.28318530717959E+0000}%
%
\special{pn 13}%
\special{sh 1}%
\special{ar 628 660 10 10 0  6.28318530717959E+0000}%
\special{sh 1}%
\special{ar 628 660 10 10 0  6.28318530717959E+0000}%
%
\special{pn 13}%
\special{sh 1}%
\special{ar 828 668 10 10 0  6.28318530717959E+0000}%
\special{sh 1}%
\special{ar 828 668 10 10 0  6.28318530717959E+0000}%
%
\special{pn 13}%
\special{sh 1}%
\special{ar 980 700 10 10 0  6.28318530717959E+0000}%
\special{sh 1}%
\special{ar 980 700 10 10 0  6.28318530717959E+0000}%
\end{picture}%
%
\raisebox{7ex}{$\rightarrow$}
\unitlength 0.1in
\begin{picture}( 12.8000, 10.4000)(  1.0000,-11.4000)
%
\special{pn 8}%
\special{ar 660 820 480 160  0.0000000 6.2831853}%
%
\special{pn 8}%
\special{pa 580 1060}%
\special{pa 580 500}%
\special{da 0.050}%
\special{pa 260 500}%
\special{pa 260 980}%
\special{da 0.050}%
\special{pa 900 1020}%
\special{pa 900 460}%
\special{da 0.050}%
%
\special{pn 8}%
\special{pa 980 380}%
\special{pa 980 780}%
\special{da 0.050}%
\special{pa 660 460}%
\special{pa 660 740}%
\special{da 0.050}%
\special{pa 340 460}%
\special{pa 340 780}%
\special{da 0.050}%
%
\special{pn 8}%
\special{pa 100 100}%
\special{pa 1380 100}%
\special{pa 1380 1140}%
\special{pa 100 1140}%
\special{pa 100 100}%
\special{ip}%
%
\special{pn 8}%
\special{pa 1212 188}%
\special{pa 812 588}%
\special{fp}%
%
\special{pn 8}%
\special{pa 1140 180}%
\special{pa 1140 1060}%
\special{fp}%
%
\special{pn 8}%
\special{pa 500 580}%
\special{pa 1300 180}%
\special{fp}%
%
\special{pn 8}%
\special{pa 180 580}%
\special{pa 1300 204}%
\special{fp}%
%
\special{pn 13}%
\special{sh 1}%
\special{ar 1140 260 10 10 0  6.28318530717959E+0000}%
\special{sh 1}%
\special{ar 1140 260 10 10 0  6.28318530717959E+0000}%
\end{picture}%
%
\raisebox{7ex}{$\rightarrow$}
\unitlength 0.1in
\begin{picture}( 16.0000, 10.4000)(  1.0000,-11.4000)
%
\special{pn 8}%
\special{pa 580 180}%
\special{pa 180 580}%
\special{fp}%
\special{pa 500 580}%
\special{pa 900 180}%
\special{fp}%
\special{pa 1220 180}%
\special{pa 820 580}%
\special{fp}%
%
\special{pn 8}%
\special{ar 660 820 480 160  0.0000000 6.2831853}%
%
\special{pn 8}%
\special{pa 580 1060}%
\special{pa 580 420}%
\special{da 0.050}%
\special{pa 260 420}%
\special{pa 260 980}%
\special{da 0.050}%
\special{pa 900 420}%
\special{pa 900 1020}%
\special{da 0.050}%
%
\special{pn 8}%
\special{pa 340 340}%
\special{pa 340 780}%
\special{da 0.050}%
\special{pa 660 340}%
\special{pa 660 740}%
\special{da 0.050}%
\special{pa 980 340}%
\special{pa 980 780}%
\special{da 0.050}%
%
\special{pn 8}%
\special{pa 340 260}%
\special{pa 1620 260}%
\special{fp}%
%
\special{pn 8}%
\special{pa 100 100}%
\special{pa 1700 100}%
\special{pa 1700 1140}%
\special{pa 100 1140}%
\special{pa 100 100}%
\special{ip}%
%
\special{pn 8}%
\special{pa 1500 180}%
\special{pa 1020 1060}%
\special{fp}%
%
\special{pn 20}%
\special{sh 1}%
\special{ar 1452 268 10 10 0  6.28318530717959E+0000}%
\special{sh 1}%
\special{ar 1452 268 10 10 0  6.28318530717959E+0000}%
%
\special{pn 8}%
\special{pa 340 276}%
\special{pa 1620 276}%
\special{fp}%
\end{picture}%
\\
\raisebox{7ex}{$\rightarrow$}
\unitlength 0.1in
\begin{picture}( 16.0000, 10.4000)(  1.0000,-11.4000)
%
\special{pn 8}%
\special{pa 580 180}%
\special{pa 180 580}%
\special{fp}%
\special{pa 500 580}%
\special{pa 900 180}%
\special{fp}%
\special{pa 1220 180}%
\special{pa 820 580}%
\special{fp}%
%
\special{pn 8}%
\special{ar 660 820 480 160  0.0000000 6.2831853}%
%
\special{pn 8}%
\special{pa 580 1060}%
\special{pa 580 420}%
\special{da 0.050}%
\special{pa 260 420}%
\special{pa 260 980}%
\special{da 0.050}%
\special{pa 900 420}%
\special{pa 900 1020}%
\special{da 0.050}%
%
\special{pn 8}%
\special{pa 340 340}%
\special{pa 340 780}%
\special{da 0.050}%
\special{pa 660 340}%
\special{pa 660 740}%
\special{da 0.050}%
\special{pa 980 340}%
\special{pa 980 780}%
\special{da 0.050}%
%
\special{pn 8}%
\special{pa 340 260}%
\special{pa 1620 260}%
\special{fp}%
%
\special{pn 8}%
\special{pa 1532 460}%
\special{pa 972 1020}%
\special{fp}%
%
\special{pn 8}%
\special{pa 100 100}%
\special{pa 1700 100}%
\special{pa 1700 1140}%
\special{pa 100 1140}%
\special{pa 100 100}%
\special{ip}%
%
\special{pn 8}%
\special{pa 1460 180}%
\special{pa 1460 660}%
\special{fp}%
%
\special{pn 13}%
\special{sh 1}%
\special{ar 1460 532 10 10 0  6.28318530717959E+0000}%
\special{sh 1}%
\special{ar 1460 532 10 10 0  6.28318530717959E+0000}%
%
\special{pn 8}%
\special{pa 1620 276}%
\special{pa 340 276}%
\special{fp}%
\end{picture}%
%
\raisebox{7ex}{$\rightarrow$}
\unitlength 0.1in
\begin{picture}( 16.0000, 10.4000)(  1.0000,-11.4000)
%
\special{pn 8}%
\special{pa 580 180}%
\special{pa 180 580}%
\special{fp}%
\special{pa 500 580}%
\special{pa 900 180}%
\special{fp}%
\special{pa 1220 180}%
\special{pa 820 580}%
\special{fp}%
\special{pa 1140 580}%
\special{pa 1540 180}%
\special{fp}%
%
\special{pn 8}%
\special{ar 660 820 480 160  0.0000000 6.2831853}%
%
\special{pn 8}%
\special{pa 580 1060}%
\special{pa 580 420}%
\special{da 0.050}%
\special{pa 260 420}%
\special{pa 260 980}%
\special{da 0.050}%
\special{pa 900 420}%
\special{pa 900 1020}%
\special{da 0.050}%
%
\special{pn 8}%
\special{pa 340 340}%
\special{pa 340 780}%
\special{da 0.050}%
\special{pa 660 340}%
\special{pa 660 740}%
\special{da 0.050}%
\special{pa 980 340}%
\special{pa 980 780}%
\special{da 0.050}%
%
\special{pn 8}%
\special{pa 1620 276}%
\special{pa 340 276}%
\special{fp}%
\special{pa 340 260}%
\special{pa 1620 260}%
\special{fp}%
%
\special{pn 8}%
\special{pa 1532 460}%
\special{pa 972 1020}%
\special{fp}%
%
\special{pn 8}%
\special{pa 1300 340}%
\special{pa 1300 820}%
\special{da 0.050}%
%
\special{pn 8}%
\special{pa 100 100}%
\special{pa 1700 100}%
\special{pa 1700 1140}%
\special{pa 100 1140}%
\special{pa 100 100}%
\special{ip}%
\end{picture}%
%
\vspace{-2ex}%
\caption{Process of resolution of cubic pencil}%
\label{Process}%
\end{figure}
\par
We can resolve the nine base points of the pencil 
by the process shown in Figure \ref{Process}. 
The first six base points are resolved by simply blowing them up, 
over which we have six $(-1)$-curves as exceptional curves. 
The remaining infinitely near base points 
are resolved by successively three-times blowing them up, 
over which we have two $(-2)$-curves and one $(-1)$-curve 
as the irreducible components of the exceptional set. 
Then this last $(-1)$-curve should be the $0$-section $s_0$ 
for our elliptic surface $\ES$ of the Weierstrass model (\ref{WModel}).
\par
By resolving the nine base points, 
the first three singular cubic curves of (\ref{Members}) become 
the singular fibers $\kt{I}_1, \kt{II}, \kt{III}$ of $\ES$ as themselves. 
The last singular cubic curve of the three concurrent lines 
becomes the fiber $\kt{I}_0^*$ of $\ES$ 
by separating the three lines and adding the two $(-2)$-curves 
arising over the infinitely near base points.
In this way, we see that 
the cubic pencil (\ref{Cubics}) determines certainly
our elliptic surface $\ES$ in the case 
$[\, \kt{I}_1, \kt{II}, \kt{III}, \kt{I}_0^*\,]$. 
\par
In the particular cases of ${\frac{b}{\,a\,} = \infty, 0, 1}$, 
two of the singular cubic curves (\ref{Members}) 
are confluent to a cubic curve that contains multiple lines. 
We have three singular cubic curves 
at ${\frac{\,\PC{T}\,}{\,\PC{S}\,} = \infty, 0, 1}$ in the members of the pencil. 
These are given by the following curves 
in the case of ${\frac{\,b\,}{a} = \infty}$: 
\begin{equation}
\left\{
\begin{array}{@{\ }l@{\quad\ }l@{\,}}
(X-3Z)(2X+3Z)^2 = 0
    & \text{(line and double line),} \\[1ex]
b\,Y^2Z = 4X^3  
    & \text{(cuspidal cubic curve),} \\[1ex]
( b\,Y^2 - 27XZ - 27 Z^2 )\,Z = 0  
    & \text{(conic and tangent line),} 
\end{array}
\right.
\end{equation}
and in the case of ${\frac{\,b\,}{a} = 0}$: 
\begin{equation}
\left\{
\begin{array}{@{\ }l@{\quad\ }l@{\,}}
a\,Y^2Z = (X-3Z)(2X+3Z)^2 
    & \text{(nodal cubic curve),} \\[1ex]
X^3 = 0  
    & \text{(triple line),} \\[1ex]
( a\,Y^2 + 27XZ + 27 Z^2 )\,Z = 0  
    & \text{(conic and tangent line),} 
\end{array}
\right.
\end{equation}
and in the case of ${\frac{\,b\,}{a} = 1}$:
\begin{equation}
\left\{
\begin{array}{@{\ }l@{\quad\ }l@{\,}}
a\,Y^2Z = (X-3Z)(2X+3Z)^2 
    & \text{(nodal cubic curve),} \\[1ex]
b\,Y^2Z = 4X^3  
    & \text{(cuspidal cubic curve),} \\[1ex]
(X+Z)\,Z^2 = 0  
    & \text{(line and double line).}
\end{array}
\right.
\end{equation}
By resolving the confluent base points, we can confirm that 
the confluent singular cubic curves 
are transformed to the singular fibers of types 
$\kt{I}_1^*, \kt{IV}^*, \kt{III}^*$. 
Then we obtain our elliptic surfaces in the cases 
$[\, \kt{I}_1^*, \kt{II}, \kt{III} \,]$, 
$[\, \kt{I}_1, \kt{IV}^*, \kt{III} \,]$, 
$[\, \kt{I}_1, \kt{II}, \kt{III}^* ]$ 
as the particular cases of the cubic pencil (\ref{Cubics}).
\par
We note that the pencil (\ref{Cubics}) in the general case 
is generated by the last two cubic curves of (\ref{Members}), 
in which the conic with the tangent line and the three concurrent lines 
are in a restricted position 
by the condition of degree-one $J$-function. 
In other words, 
a similar pencil generated by such cubic curves 
in more general position are not in ours, 
though the base points of it would be resolved 
by the same process as shown in Figure \ref{Process}. 
The obtained surface is a rational elliptic surface 
with $J$-function of higher degree $({>}\,1)$.
In the cubic form (\ref{CubicsWform}) of our pencil, 
the coefficients as linear forms with respect to ${(\PC{S}:\PC{T})}$
are in rather restricted conditions. 
The coefficient of $X^2 Z$ is made to be entirely zero. 
Furthermore, the coefficients of $XZ^2$ and $Z^3$ have a common factor,  
by which we have the $J$-function to be of degree one. 
\section{Sections of elliptic fibrations of degree one}\label{Sec:3}%
We will describe sections as divisor classes
for our rational elliptic surface $\ES$. 
We mainly discuss them in the general case 
$[\, \kt{I}_1, \kt{II}, \kt{III}, \kt{I}_0^*\,]$. 
\par
The fibers of types $\kt{I}_1, \kt{II}$ are 
irreducible singular rational curves with a node, a cusp, respectively. 
These divisor classes are the same 
to the fiber class $f$ of the fibration, 
and the $0$-section $s_0$ meets both of them. 
The fiber of type $\kt{III}$ consists 
of two tangential smooth rational curves, 
which we denote by $u_0, u_1$ such that $s_0$ meets $u_0$. 
The fiber of type $\kt{I}_0^*$ consists 
of five smooth rational curves 
$v_0, v_1, v_2, v_3$ (with multiplicity $1$) 
and $v_4$ (with multiplicity $2$), 
which form the intersection behavior 
of the extended diagram of type $\widetilde{\RS{D}}_4$, 
where $v_4$ is the central component of $\widetilde{\RS{D}}_4$ 
and $s_0$ meets $v_0$. 
The fiber class $f$ is written by the components 
of the fibers $\kt{III}$ and $\kt{I}_0^*$ as follows:
\begin{equation*}
f = u_0 + u_1 \ \ \ \text{and}\ \ \ f = v_0 + v_1 + v_2 + v_3 + 2v_4.
\end{equation*}
\par
The N\'eron-Severi group $\NS(\ES)$ is generated 
by the orthogonal classes $\dcl, e_1, \cdots{}, e_9$ 
associated with the resolution of the (ordered) nine base points 
for the pencil (\ref{Cubics}) of plane cubic curves. 
We now consider an order for the nine base points such that 
(three pairs of) 
$1,4$-th base points, $2,5$-th base points, $3,6$-th base points 
are, respectively, on one of the three lines 
of the last singular cubic curve of (\ref{Members}), 
and the $7,8,9$-th base points are 
the infinitely near base points of the pencil (\ref{Cubics}).
We see that, under this order, 
the components of the singular fibers 
$\kt{III}$ and $\kt{I}_0^*$ 
are written by the following:
\begin{equation*}
\begin{array}{@{}l@{}}
u_0 = \dcl - e_7 - e_8 - e_9, \\[.5ex]
u_1 = 2\dcl - e_1 - e_2 - e_3 - e_4 - e_5 - e_6, 
\end{array} \ \ \ \ \ 
\begin{array}{@{}l@{}}
v_0 = e_8 - e_9, \\[.5ex]
v_4 = e_7 - e_8, 
\end{array} \ \ \ \ \ 
\begin{array}{@{}l@{}}
v_1 = \dcl - e_1 - e_4 - e_7, \\[.5ex]
v_2 = \dcl - e_2 - e_5 - e_7, \\[.5ex]
v_3 = \dcl - e_3 - e_6 - e_7.
\end{array}
\end{equation*}
For the exceptional curves $e_1, \cdots{}, e_9$ 
appearing over the base points, 
the first six curves $e_1, \cdots{}, e_6$ and the last curve $e_9$ 
become mutually disjoint irreducible sections 
with the $0$-section $s_0 = e_9$. 
However, $e_7, e_8$ are not irreducible sections, 
which contain fiber components 
(as presented in the proof of the following proposition). 
We have that 
the component $u_1$ of the fiber $\kt{III}$ 
meets $e_1, \cdots, e_6$, 
and $v_1$ of $\kt{I}_0^*$ meets $e_1, e_4$, 
and $v_2$ of $\kt{I}_0^*$ meets $e_2, e_5$, 
and $v_3$ of $\kt{I}_0^*$ meets $e_3, e_6$, 
while $s_0 = e_9$ meets $u_0, v_0$ 
(as shown in Figure \ref{Configuration}). 
\begin{figure}[htb]
\centering\normalsize
\input{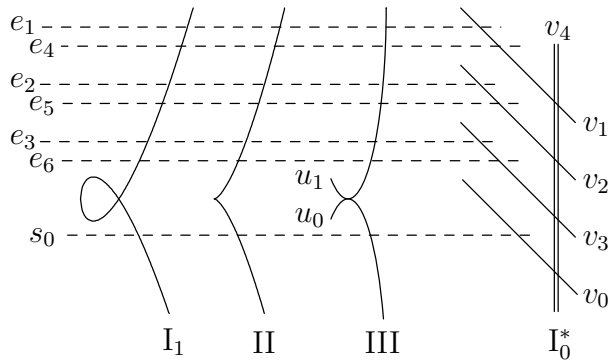}%
\vspace{1ex}%
\caption{Configuration of singular fibers and sections}%
\label{Configuration}%
\end{figure}
\par
The Mordell-Weil group $\MW(\ES)$ of irreducible sections 
is isomorphic to the factor group $\NS(\ES) / (U \oplus L)$ 
as mentioned in Section \ref{Sec:1}, 
where ${U = \langle \,f,s_0\, \rangle}$ and 
$L$ is the sublattice in $\NS(\ES)$ generated 
by the fiber components not meeting $s_0$.
In our case of $[\, \kt{I}_1, \kt{II}, \kt{III}, \kt{I}_0^*\,]$, 
the fiber components $u_1$ and $v_1, \cdots{}, v_4$ 
form a root basis of the sublattice $L$ 
of type ${\RS{A}_1 \oplus \RS{D}_4}$ with rank $5$, 
while $L$ becomes of type 
${\RS{A}_1 \oplus \RS{D}_5}$, ${\RS{A}_1 \oplus \RS{E}_6}$, $\RS{E}_7$ 
with rank $6, 7, 7$ in the particular cases 
$[\, \kt{I}_1^*, \kt{II}, \kt{III} \,]$, 
$[\, \kt{I}_1, \kt{IV}^*, \kt{III} \,]$, 
$[\, \kt{I}_1, \kt{II}, \kt{III}^* ]$, respectively. 
These types of sublattices are primitive 
in the root lattice of type $\RS{E}_8$, 
which gives that $\MW(\ES)$ is torsion-free, 
and we have the following:
%
%
\begin{Prop}
For the rational elliptic surface $\ES$ of the cubic pencil $(\ref{Cubics})$, 
the Mordell-Weil group $\MW(\ES)$ is a free $\Z$-module 
generated by irreducible sections of exceptional curves 
appearing over the base points of the pencil. 
The rank of\/ $\MW(\ES)$ is equal to $3$ 
in the general case $[\, \kt{I}_1, \kt{II}, \kt{III}, \kt{I}_0^*\,]$, 
while it is equal to $2,1,1$ in the particular cases 
$[\, \kt{I}_1^*, \kt{II}, \kt{III} \,]$, 
$[\, \kt{I}_1, \kt{IV}^*, \kt{III} \,]$, 
$[\, \kt{I}_1, \kt{II}, \kt{III}^* ]$, respectively.
\end{Prop}
\begin{proof}
In the general case $[\, \kt{I}_1, \kt{II}, \kt{III}, \kt{I}_0^*\,]$, 
the fiber components not meeting $s_0$ generate 
the sublattice $L$ of type ${\RS{A}_1 \oplus \RS{D}_4}$ with rank $5$. 
Thus, the rank of $\MW(\ES)$ is equal to $3$ 
by the Shioda-Tate formula (\ref{Shioda-Tate}).
The divisor classes $\dcl, e_7, e_8$ and ${e_1+e_4}$, ${e_2+e_5}$, ${e_3+e_6}$ 
are represented by the $0$-section $s_0$ and fiber components as follows:
\begin{equation*}
\begin{array}{@{}c@{\,\,}c@{\,\,}l@{}}
\dcl &=& 3s_0 + u_0 + 2v_0 + v_4, \\[.5ex]
e_7 &=& s_0 + v_0 + v_4, \\[.5ex]
e_8 &=& s_0 + v_0, 
\end{array} \ \ \ \ \ \ \ 
\begin{array}{@{}c@{\,\,}c@{\,\,}l@{}}
e_1 + e_4 &=& 2s_0 + u_0 + v_0 - v_1, \\[.5ex]
e_2 + e_5 &=& 2s_0 + u_0 + v_0 - v_2, \\[.5ex]
e_3 + e_6 &=& 2s_0 + u_0 + v_0 - v_3.
\end{array}
\end{equation*}
These are considered to be zero as residue class 
in $\NS(\ES) / (U \oplus L)$.
Therefore, the irreducible sections of the exceptional curves $e_1, e_2, e_3$ 
form a basis of $\MW(\ES)$. 
We can similarly determine the rank of $\MW(\ES)$ with its generators 
in the remaining particular cases 
$[\, \kt{I}_1^*, \kt{II}, \kt{III} \,]$, 
$[\, \kt{I}_1, \kt{IV}^*, \kt{III} \,]$, 
$[\, \kt{I}_1, \kt{II}, \kt{III}^* ]$. 
\end{proof}
%
%
\par
According to Oguiso-Shioda's list of \cite{OS}, 
the Mordell-Weil groups of our elliptic surfaces 
have the following lattice structures 
(corresponding to No.$18$, $30$, $49$, $43$).
\begin{equation*}
\begin{array}{@{\ \ }c@{\ \ } *{3}{| @{\ \ }c@{\ \ }} }
[\, \kt{I}_1, \kt{II}, \kt{III}, \kt{I}_0^*\,]
  & [\, \kt{I}_1^*, \kt{II}, \kt{III} \,]
    & [\, \kt{I}_1, \kt{IV}^*, \kt{III} \,]
      & [\, \kt{I}_1, \kt{II}, \kt{III}^* ] \\
\hline
A_1^* \oplus A_1^* \oplus A_1^* 
 & A_1^* \oplus \langle \frac{\,1\,}{4} \rangle 
 & \langle \frac{\,1\,}{6} \rangle 
 & A_1^* \\
\end{array}
\end{equation*}
In the discussion below, 
we see the lattice structure of the general case 
$[\, \kt{I}_1, \kt{II}, \kt{III}, \kt{I}_0^*\,]$ 
rather directly by representing sections 
as divisor classes in $\NS(\ES)$. 
\par
To avoid confusion, 
we denote the addition of $\MW(\ES)$ by $\dot{+}\,$ 
and denote the irreducible sections of the exceptional curves 
$e_1, e_2, e_3$ by $s(e_1), s(e_2), s(e_3)$. 
Any irreducible section ${s \in \MW(\ES)}$ is represented 
by a linear combination of the base sections $s(e_1), s(e_2), s(e_3)$ 
with integer coefficients:
\begin{equation*}
{}\ \ \ \ \ \ 
s = n_1 s(e_1) \,\dot{+}\, n_2 s(e_2) \,\dot{+}\, n_3 s(e_3)
\ \ \ \ \ \ (\, n_1, n_2, n_3 \in \Z \,).
\end{equation*}
The section $s$ determines a divisor class in $\NS(\ES)$, 
and it should be represented 
by a linear combination of the base classes 
$\dcl, e_1, \cdots{}, e_9$ of $\NS(\ES)$.
We present this in the following proposition:
%
%
\begin{Prop}
For the rational elliptic surface $\ES$ in the case 
$[\, \kt{I}_1, \kt{II}, \kt{III}, \kt{I}_0^*\,]$ 
with base sections $s(e_1), s(e_2), s(e_3)$ of\/ $\MW(\ES)$ 
$($determined in association with 
base points of the cubic pencil $(\ref{Cubics})\,)$, 
the irreducible section ${s \in \MW(\ES)}$ written 
by the linear combination of $s(e_1), s(e_2), s(e_3)$ 
with coefficients $n_1, n_2, n_3 \,({\in}\,\Z)$ 
is represented by the base classes 
$\dcl, e_1, \cdots{}, e_9$ of\/ $\NS(\ES)$ 
as the following divisor class\/$:$
\begin{equation*}
\begin{array}{@{}c@{\,\,}c@{\,}l@{}}
s &=& \big( 3m + [\frac{\,\varepsilon\,}{2}](\varepsilon-1) \big) \dcl \\[.75ex]
  & & {} - \big( m - \frac{n_1 + \varepsilon_1}{2} 
                   + [\frac{\,\varepsilon\,}{2}] \varepsilon_1 \big) e_1
         - \big( m - \frac{n_2 + \varepsilon_2}{2} 
                   + [\frac{\,\varepsilon\,}{2}] \varepsilon_2 \big) e_2
         - \big( m - \frac{n_3 + \varepsilon_3}{2} 
                   + [\frac{\,\varepsilon\,}{2}] \varepsilon_3 \big) e_3 \\[.75ex]
  & & {} - \big( m + \frac{n_1 - \varepsilon_1}{2} 
                   + [\frac{\,\varepsilon\,}{2}] \varepsilon_1 \big) e_4
         - \big( m + \frac{n_2 - \varepsilon_2}{2} 
                   + [\frac{\,\varepsilon\,}{2}] \varepsilon_2 \big) e_5
         - \big( m + \frac{n_3 - \varepsilon_3}{2} 
                   + [\frac{\,\varepsilon\,}{2}] \varepsilon_3 \big) e_6 \\[.75ex]
  & & {} - \big( m + [\frac{\,\varepsilon\,}{2}] (\varepsilon -2) \big) e_7
         - \big( m + [\frac{\,\varepsilon\,}{2}] (\varepsilon -2) \big) e_8
         - \big( m + (1 - [\frac{\,\varepsilon\,}{2}] )(\varepsilon -1) \big) e_9, 
\end{array}
\end{equation*}
where $\varepsilon_i$ is the remainder of the division of\/ $n_i$ by $2$, 
and $\varepsilon = \varepsilon_1 + \varepsilon_2 + \varepsilon_3$, 
and $[\frac{\,\varepsilon\,}{2}]$ 
is the quotient of the division of\/ $\varepsilon$ by $2$, 
and $m = \frac{\,1\,}{4} 
         (n_1^{\,2} + n_2^{\,2} + n_3^{\,2} - \varepsilon ) \,(\in \Z)$.
\end{Prop}
\begin{proof}
We will show that the divisor class $s$ 
as represented above by $\dcl, e_1, \cdots{}, e_9$ 
coincides with the linear combination of $s(e_1), s(e_2), s(e_3)$. 
The representation of $s$ is the same to the following one:
\begin{equation*}
\begin{array}{@{}c@{\,\,}c@{\,}l@{}}
s &=& \,e_9 + n_1 (e_1-e_9) + n_2 (e_2-e_9) + n_3 (e_3-e_9) \\[.75ex]
  & & {} - \frac{n_1 - \varepsilon_1}{2} (u_0+v_0-v_1)  
         - \frac{n_2 - \varepsilon_2}{2} (u_0+v_0-v_2) 
         - \frac{n_3 - \varepsilon_3}{2} (u_0+v_0-v_3) \\[.75ex]
  & & {} - [\frac{\,\varepsilon\,}{2}] 
           \big( u_0 + (\varepsilon - 2)v_0 - \varepsilon_1 v_1 
                 - \varepsilon_2 v_2 - \varepsilon_3 v_3 - v_4 \big) + mf.
\end{array}
\end{equation*}
Hence, our assertion follows at least up to fiber components.
It is not hard to check that 
$s$ satisfies $(s \cdot f) = 1$ and $(s \cdot s) = -1$ 
as the condition of numerical sections. 
In fact, the fiber components are orthogonal to the fiber class $f$, 
and $[\frac{\,\varepsilon\,}{2}]^2 = [\frac{\,\varepsilon\,}{2}]$
for $\varepsilon = 0,1,2,3$, 
and then we can check the following:
\begin{equation*}\textstyle
(s \cdot f) = (e_9 \cdot f) = 1 
\ \ \ \text{and}\ \ \ 
(s \cdot s) = [\frac{\,\varepsilon\,}{2}] (4 \varepsilon -6)
              - \varepsilon (\varepsilon -1) -1 
            = -1.
\end{equation*}
To show that $s$ is irreducible, 
it is sufficient to check that, 
on each (reducible) singular fiber, 
$s$ intersects only one fiber component of multiplicity $1$. 
For the components of the fiber $\kt{III}$, 
we have the following intersection numbers:
\begin{equation*}\textstyle
(s \cdot u_0) = 1 - \varepsilon + 2 [\frac{\,\varepsilon\,}{2}] 
\ \ \ \text{and}\ \ \ 
(s \cdot u_1) = \varepsilon - 2 [\frac{\,\varepsilon\,}{2}],
\end{equation*}
which mean that $s$ intersects either $u_0$ or $u_1$. 
For the components of the fiber $\kt{I}_0^*$, 
we have $(s \cdot v_4) = 0$ 
with the component $v_4$ of multiplicity $2$, 
and we have the following intersection numbers 
with the other components of multiplicity $1$:
\begin{equation*}\textstyle
(s \cdot v_0) = 1 - \varepsilon + [\frac{\,\varepsilon\,}{2}] (2 \varepsilon -3) 
\ \ \ \text{and}\ \ \ 
\begin{array}{@{}l@{}}
(s \cdot v_1) = \varepsilon_1 + [\frac{\,\varepsilon\,}{2}] (1 - 2\varepsilon_1), \\[.75ex]
(s \cdot v_2) = \varepsilon_2 + [\frac{\,\varepsilon\,}{2}] (1 - 2\varepsilon_2), \\[.75ex]
(s \cdot v_3) = \varepsilon_3 + [\frac{\,\varepsilon\,}{2}] (1 - 2\varepsilon_3),
\end{array}
\end{equation*}
which mean that $s$ intersects 
only one fiber component of multiplicity $1$ of the fiber $\kt{I}_0^*$. 
Therefore, $s$ is the desired irreducible section.
\end{proof}
%
%
\par
The above representation of the section $s$ 
by $\dcl, e_1, \cdots{}, e_9$ means that, 
by the process of the resolution of the cubic pencil (\ref{Cubics}), 
the section $s$ is obtained 
by the proper transform of an algebraic curve on $\PS[2]$ 
whose degree is given by the coefficient of $\dcl$, 
and this curve on $\PS[2]$ passes through the $i$-th base point 
with the multiplicity given 
by the minus of the coefficient of $e_i$. 
For instance, the section $s(e_1) \,\dot{+}\, s(e_2)$ 
is represented as follows:
\begin{equation*}
{}\hspace{-4.75ex}
\begin{array}{@{}r@{\,\,}c@{\,\,}l@{}}
s(e_1) \,\dot{+}\, s(e_2) 
 &=& \dcl - e_4\! - e_5 \\[.5ex]
 &{}\!\!\!\big(\,{=}& e_9\! + (e_1\! -e_9) + (e_2\! -e_9) 
                      + f - u_0\! - v_0\! - v_3\! - v_4 \,\big),
\end{array}
\end{equation*}
which means that this section is the proper transform of the line 
passing through once the $4,5$-th base points. 
\par
We now transform the representation of the section $s$ 
by the following relations:
\begin{equation*}\textstyle
\begin{array}{@{}l@{}}
e_1 - e_9 = \frac{\,1\,}{2} (e_1 - e_4) 
            + \frac{\,1\,}{2} (e_1 + e_4 - 2e_9) 
          = \frac{\,1\,}{2}\,\rt_{14} 
            + \frac{\,1\,}{2} (u_0 + v_0 - v_1), \\[.75ex]
e_2 - e_9 = \frac{\,1\,}{2} (e_2 - e_5) 
            + \frac{\,1\,}{2} (e_2 + e_5 - 2e_9) 
          = \frac{\,1\,}{2}\,\rt_{25} 
            + \frac{\,1\,}{2} (u_0 + v_0 - v_2), \\[.75ex]
e_3 - e_9 = \frac{\,1\,}{2} (e_3 - e_6) 
            + \frac{\,1\,}{2} (e_3 + e_6 - 2e_9) 
          = \frac{\,1\,}{2}\,\rt_{36} 
            + \frac{\,1\,}{2} (u_0 + v_0 - v_3).
\end{array}
\end{equation*}
Thus, we obtain the following representation 
(with rational coefficients):
\begin{equation*}
\begin{array}{@{}c@{\,\,}c@{\,}l@{}}
s &=& \,s_0 + \frac{\,n_1}{2}\,\rt_{14} + \frac{\,n_2}{2}\,\rt_{25} + \frac{\,n_3}{2}\,\rt_{36} \\[.75ex]
  & & {} + \frac{\,\varepsilon_1}{2} (u_0+v_0-v_1)  
         + \frac{\,\varepsilon_2}{2} (u_0+v_0-v_2) 
         + \frac{\,\varepsilon_3}{2} (u_0+v_0-v_3) \\[.75ex]
  & & {} - [\frac{\,\varepsilon\,}{2}] 
           \big( u_0 + (\varepsilon - 2)v_0 - \varepsilon_1 v_1 
                 - \varepsilon_2 v_2 - \varepsilon_3 v_3 - v_4 \big) + mf.
\end{array}
\end{equation*}
Three roots $\rt_{14}, \rt_{25}, \rt_{36}$ 
are orthogonal to each other, 
which form the root system of type $A_1 \oplus A_1 \oplus A_1$. 
These roots are also orthogonal 
to the fiber components and the $0$-section, 
and thus they are contained in $(U \oplus L)^{\bot}$. 
As mentioned in Section \ref{Sec:1}, 
the lattice structure of $\MW(\ES)$ 
is given by the dual lattice of $(U \oplus L)^{\bot}$. 
By the above representation of the section $s$, 
this dual lattice is identified 
with the lattice generated by 
$\frac{\,1\,}{2} \rt_{14}, 
 \frac{\,1\,}{2} \rt_{25}, 
 \frac{\,1\,}{2} \rt_{36}$.  
Thus, in the general case 
$[\, \kt{I}_1, \kt{II}, \kt{III}, \kt{I}_0^*\,]$, 
the lattice structure of $\MW(\ES)$ 
is certainly of type $A_1^* \oplus A_1^* \oplus A_1^* $. 
\section{Coordinate representation of sections}\label{Sec:4}%
In this section, 
we will represent generators of the sections 
for our rational elliptic surface (\ref{WModel}) 
in terms of the coordinates 
${(\WC{X} : \WC{Y} : \WC{Z})}$ of the Weierstrass model.
By the birational correspondence of (\ref{Transform}), 
it is equivalent to describe 
the nine base points of the cubic pencil (\ref{Cubics}) 
in terms of the coordinates ${(X:Y:Z)}$ of $\PS[2]$.
\par
To determine the base points, 
we will solve (\ref{Cubics}) (equivalently (\ref{Members})) 
as simultaneous cubic equations. 
In the general case $[\, \kt{I}_1, \kt{II}, \kt{III}, \kt{I}_0^*\,]$, 
as mentioned in Section \ref{Sec:2}, 
three base points are infinitely near 
to the inflection point ${(0:1:0)}$ associated with the $0$-section. 
The other six base points 
are particularly on the conic $({\cong}\ \PS[1])$ 
in the third singular cubic curve of (\ref{Members}), 
and they are divided into three pairs 
lying, respectively, on one of the three lines 
in the last singular cubic curve of (\ref{Members}). 
This situation may remind us of some association with a octahedron, 
that is, the three pairs for the base points 
may be associated with the three antipodal pairs 
for the vertices of the octahedron. 
Hence, we use a octahedral covering space in the discussion below, 
for which the octahedral group of the covering transformations 
would be interpreted as some permutations of the base points.
\par
We will construct a homogeneous mapping 
from $\C^2$ with coordinates ${(\alpha, \beta)}$ 
onto $\C^2$ with ${(a, b)}$ 
such that it induces a (branched) covering mapping 
from $\PS[1]$ with homogeneous coordinates ${(\alpha : \beta)}$ 
onto $\PS[1]$ with ${(a:b)}$ concerning action of the octahedral group. 
We consider here the regular octahedron to be embedded 
into the Riemann sphere ${\PS[1] = \C \cup \{ \infty \}}$ 
with the six vertices
as ${\frac{\alpha}{\,\beta\,} = \infty, 0, \pm \,1, \pm \iu}$ 
${(\iu \!=\! \sqrt{-1}\,)}$. 
Then the action of the octahedral group 
on $\PS[1]$ with ${(\alpha : \beta)}$ is lifted 
as the unitary action of the binary octahedral group 
on $\C^2$ with ${(\alpha, \beta)}$. 
\par
We recall several invariant polynomials 
with respect to the group action 
(see \cite{Mukai}, \cite{Spr} for details of invariant theory).
Let $V, E, F$ be the following homogeneous polynomials:
\begin{equation*}
\begin{array}{@{}l@{}}
V = 4\,\alpha \beta (\alpha^4 - \beta^4),  \\[1ex]
E = 8 (\alpha^4 + \beta^4)(\alpha^8 - 34\,\alpha^4 \beta^4 + \beta^8),  \\[1ex]
F = 4 (\alpha^8 + 14\,\alpha^4 \beta^4 + \beta^8), 
\end{array}
\end{equation*}
(where the coefficients $4, 8, 4$ of $V, E, F$ are attached for later discussion). 
In the Riemann sphere, 
the zeros of $V, E, F$ are given as follows:
\begin{equation*}
\begin{array}{@{}ccl@{}}
V=0 & \leftrightarrow & \textstyle
      \frac{\alpha}{\,\beta\,} = 
      \infty,\, 0,\, \pm\,1,\, \pm \iu, \\[1ex]
E=0 & \leftrightarrow & \textstyle
      \frac{\alpha}{\,\beta\,} = 
      {\pm} \frac{1}{\sqrt{2}}(1 \pm \iu),\, 
      {\pm} (\sqrt{2} \pm 1),\, 
      {\pm} (\sqrt{2} \pm 1) \iu,  \\[1ex]
F=0 & \leftrightarrow & \textstyle
      \frac{\alpha}{\,\beta\,} = 
      {\pm} \frac{1 + \iu}{\sqrt{3}-1},\, 
      {\pm} \frac{1 - \iu}{\sqrt{3}-1},\,
      {\pm} \frac{1 + \iu}{\sqrt{3}+1},\,
      {\pm} \frac{1 - \iu}{\sqrt{3}+1}.
\end{array}
\end{equation*}
They correspond to the vertices, 
the middle points of the edges, the central points of the faces 
in the regular octahedron, respectively.
Under the action of the binary octahedral group, 
$F$ is invariant, but $V$ and $E$ may change their signs themselves. 
These are invariant under the action of the normal subgroup of index $2$
as the binary tetrahedral group.
\par
The invariant ring of the polynomials 
under the tetrahedral action is generated by $V, F, E$, 
while the invariant ring 
under the octahedral action is generated by $F, V^2, V\! E$. 
These generators have the following relations:
\begin{equation*}
E^2 = F^3 - 27 V^4 
\ \ \ \text{and}\ \ \ 
(V\! E)^2 = (F)^3 (V^2) - 27 (V^2)^3,
\end{equation*}
(which are associated with the rational double singularities 
of types $\RS{E}_6, \RS{E}_7$). 
In the octahedral invariant ring, 
the homogeneous part of degree $24$ 
is the vector space of dimension $2$ 
generated by (two of) the invariants $E^2, F^3, V^4$ 
with the above (linear) relation among them. 
By setting ${(a,b) = (E^2, F^3)}$, 
we obtain a homogeneous mapping 
from $\C^2$ with ${(\alpha, \beta)}$ onto $\C^2$ with ${(a,b)}$, 
and this is the mapping mentioned above for our purpose. 
\par
By the homogeneous mapping, we have the $24$-to-$1$ covering mapping 
from $\PS[1]$ with ${(\alpha : \beta)}$ onto $\PS[1]$ with ${(a : b)}$ 
concerning the octahedral group action. 
This covering is  branched at the points 
${\frac{\,b\,}{a} = \infty, 0, 1}$, 
over which we have the points as the zeros of $E, F, V$ 
associated with the $12$ edges, the $8$ faces, the $6$ vertices 
of the octahedron, respectively.
As an intermediate covering space, 
we have $\PS[1]$ with ${(E : V^2)}$ determined 
by the basis $E, V^2$ of the homogeneous part of degree $12$ 
in the tetrahedral invariant ring. 
This covering space is the quotient space 
of $\PS[1]$ with ${(\alpha : \beta)}$ 
by the tetrahedral group action 
while it gives the double covering space 
over $\PS[1]$ with ${(a:b) = (E^2 : F^3)}$ 
branched at ${\frac{\,b\,}{a} = \infty, 1}$ by the relation below:
\begin{equation*}
{}\hspace{4ex}
\Big( \nfrac{V^2\!}{E} \Big)^{\!2} 
 = \nfrac{1}{27} \Big( \nfrac{F^3}{E^2} - 1 \Big)
\ \ \ \ \ 
\Big( \,{\Leftrightarrow} \ E^2 = F^3 - 27 V^4 \,\Big).
\end{equation*}
The octahedral group has the normal subgroup of index $6$ 
as the Klein four-group being the dihedral group for a regular $2$-gon 
(with the quaternion group as its binary version). 
The quotient space by this group action 
gives another intermediate covering space, 
and we have then the following sequence of covering spaces:
\begin{equation*}
\begin{array}{@{}c@{}c@{}c@{}c@{}c@{}c@{}c@{}}
\PS[1] & \xrightarrow{\ 4 \,:\,1 \ } & \PS[1] 
       & \xrightarrow{\ 3 \,:\,1 \ } & \PS[1] 
       & \xrightarrow{\ 2 \,:\,1 \ } & \PS[1] \phantom{= (a:b)} \\[1ex]
(\alpha : \beta)  & & (V_1^{\,2} : V_2^{\,2}) 
                  & & (E : V^2)  & & (E^2 : F^3) = (a:b),
\end{array}
\end{equation*}
where $V_1, V_2$ are certain polynomials of degree $2$, 
which are defined with $V_3$ later 
in the discussion to describe the sections of our elliptic surfaces. 
The invariant ring for the action of the quaternion group 
is generated by $V_1^{\,2}, V_2^{\,2}$ and $V$. 
This invariant ring has the homogeneous part of degree $4$ 
as the vector space of dimension $2$ 
generated by $V_1^{\,2}, V_2^{\,2}$.
For a detailed description, see Appendix of this paper.
\par
For the affine parameter ${(\alpha, \beta)}$ 
over ${(a, b) \ne (0,0)}$, 
we have the Weierstrass model (\ref{WModel}) with ${(a,b) = (E^2, F^3)}$, 
and it is now written by the following form:
\begin{equation}\label{WModel2}
\WC{Y}^2\WC{Z} 
= 4\WC{X}^3 - g_2 \WC{XZ}^2 -g_3 \WC{Z}^3 \ \ \ 
\Bigg(
\begin{array}{@{\,}l@{}}
g_2 = 27\,\PC{T}\,(\PC{T} - \PC{S}) (E^2 \PC{T} - F^3 \PC{S})^2 \\[1ex]
g_3 = 27\,\PC{T}\,(\PC{T} - \PC{S})^2 (E^2 \PC{T} - F^3 \PC{S})^3
\end{array}
\Bigg),
\end{equation}
which is generally in the case 
$[\, \kt{I}_1, \kt{II}, \kt{III}, \kt{I}_0^*\,]$ 
and particularly in the cases 
$[\, \kt{I}_1^*, \kt{II}, \kt{III} \,]$, 
$[\, \kt{I}_1, \kt{IV}^*, \kt{III} \,]$, 
$[\, \kt{I}_1, \kt{II}, \kt{III}^* ]$ 
depending on the vanishing points of the polynomials $E, F, V$ 
as shown below.
\begin{equation*}
\begin{array}{@{\ \ }c@{\ \ } *{3}{| @{\ \ }c@{\ \ }} }
[\, \kt{I}_1, \kt{II}, \kt{III}, \kt{I}_0^*\,]
  & [\, \kt{I}_1^*, \kt{II}, \kt{III} \,]
    & [\, \kt{I}_1, \kt{IV}^*, \kt{III} \,]
      & [\, \kt{I}_1, \kt{II}, \kt{III}^* ] \\
\hline
E,F,V \ne 0 & E=0 & F=0 & V=0 
\end{array}
\end{equation*}
\par
The above Weierstrass model determines 
our family of rational elliptic surfaces 
over the whole octahedral covering space $\PS[1]$ 
with the projectified parameter ${(\alpha : \beta)}$ 
(lying over the intermediate 
double covering space $\PS[1]$ with ${(E : V^2)}$).
Our family is locally represented on the points of $E \ne 0$ 
by the following Weierstrass form:
\begin{equation*}
\WC{Y}_1^{\,2}\WC{Z}_1^{} 
 = 4\WC{X}_1^{\,3} 
   - \nfrac{{g\mathstrut}_2}{E^4} \WC{X}_1^{} \WC{Z}_1^{\,2} 
   - \nfrac{{g\mathstrut}_3}{E^6} \WC{Z}_1^{\,3} 
\ \ \text{with}\ \ 
(\WC{X}_1 : \WC{Y}_1 : \WC{Z}_1) 
= \Big( \nfrac{\WC{X}}{\,E^2\!} 
        : \nfrac{\WC{Y}}{\,E^3\!} 
        : \nfrac{\WC{Z}}{1\mathstrut} \Big),
\end{equation*}
and on the points of $V^2 \ne 0$ by the following Weierstrass form:
\begin{equation*}
\WC{Y}_2^{\,2}\WC{Z}_2^{} 
 = 4\WC{X}_2^{\,3} 
   - \nfrac{{g\mathstrut}_2}{V^8} \WC{X}_2^{} \WC{Z}_2^{\,2} 
   - \nfrac{{g\mathstrut}_3}{V^{12}\!} \WC{Z}_2^{\,3} 
\ \ \text{with}\ \ 
(\WC{X}_2 : \WC{Y}_2 : \WC{Z}_2) 
= \Big( \nfrac{\WC{X}}{\,V^4\!} 
        : \nfrac{\WC{Y}}{\,V^6\!} 
        : \nfrac{\WC{Z}}{1\mathstrut} \Big).
\end{equation*}
These give locally trivialized versions of our family 
with respect to the parameter ${(\alpha : \beta)}$ (over ${(E : V^2)}$), 
and they are patched up by the relation below:
\begin{equation*}\textstyle
\WC{X}_2 = \Big( \nfrac{V^2\!}{E} \Big)^{\!\!-2} \WC{X}_1, \ \ 
\WC{Y}_2 = \Big( \nfrac{V^2\!}{E} \Big)^{\!\!-3} \WC{Y}_1, \ \   
\WC{Z}_2 = \WC{Z}_1.
\end{equation*}
Every elliptic surface in our family has the same base curve 
as the projective line $\PS[1]$ with ${(\PC{S}:\PC{T})}$ 
canonically parameterized by the $J$-function 
${J = \frac{\,\PC{T}\,}{\,\PC{S}\,}}$, 
which is independent of the parameter ${(\alpha : \beta)}$ of the family. 
Thus, the base curves of our elliptic surfaces 
make the trivial $\PS[1]$-bundle over the parameter of the family, 
and then ${(\WC{X} : \WC{Y} : \WC{Z})}$ 
should be interpreted as the global coordinate system 
of the $\PS[2]$-bundle 
${\PS(\LB{L}^2 \oplus \LB{L}^3 \oplus \LB{O}_{\PS[1] \times \PS[1]})}$
determined by a suitable fundamental line bundle $\LB{L}$ 
over the trivial bundle ${\PS[1] \times \PS[1]}$. 
It would be sufficient to set 
${\LB{L} = \LB{O}_{\PS[1] \times \PS[1]} (1, 12)}$ 
associated with the bihomogeneous forms of bidegree $(1,12)$ 
with respect to ${(\PC{S}:\PC{T})}$ and ${(\alpha : \beta)}$ 
because the Weierstrass coefficients $g_2, g_3$ are homogeneous forms 
of degree $4,6$ for ${(\PC{S}:\PC{T})}$ 
and of degree ${48 \,({=}\, 4 \cdot 12)}$, ${72 \,({=}\, 6 \cdot 12)}$ 
for ${(\alpha : \beta)}$, respectively, 
(and then $g_2, g_3$ are considered 
to be global sections of $\LB{L}^4, \LB{L}^6$).
\par
Now, we will describe sections 
for our rational elliptic surfaces of (\ref{WModel2}).
For this, we set $V_1, V_2, V_3$ to be 
the following polynomials of degree $2$ (with the relation):
\begin{equation*}
V_1 = -2\iu \alpha \beta, \ \ 
V_2 = \iu (\alpha^2 - \beta^2), \ \ 
V_3 = \alpha^2 + \beta^2 \ \ \ 
(\,V_1^{\,2} + V_2^{\,2} + V_3^{\,2} = 0\,).
\end{equation*}
They are associated with the antipodal pairs 
${(\infty, 0)}$, ${(1, -1)}$, ${(\iu, -\iu)}$ 
for the vertices of the octahedron, respectively. 
They are permuted up to sign 
by the action of the (binary) octahedral group, 
and the polynomials $V, E, F$ are represented by them as follows:
\begin{equation*}
\begin{array}{@{}l@{}}
V = 2\,V_1 V_2 V_3, \\[1ex]
E = 4 (V_1^{\,2} - V_2^{\,2})
       (V_1^{\,2} - V_3^{\,2})
        (V_2^{\,2} - V_3^{\,2}), \\[1ex]
F = - 4 (V_1^{\,2} V_2^{\,2} 
         + V_1^{\,2} V_3^{\,2} 
          + V_2^{\,2} V_3^{\,2}).
\end{array}
\end{equation*}
We have here the following main result:
%
%
\begin{Thm}\label{MainThm}
For the rational elliptic surfaces of the family $(\ref{WModel2})$, 
their generators of the sections are given 
by the following rational representation\/$:$ 
\begin{equation}\label{Generator}
\nfrac{\WC{X}}{\WC{Z}}
= \nfrac{F(\PC{T} - \PC{S})(E^2 \PC{T} - F^3 \PC{S})}{4(V_iV_j)^2}, \ \ 
\nfrac{\WC{Y}}{\WC{Z}}
= \nfrac{(\PC{T} - \PC{S})(E^2 \PC{T} - F^3 \PC{S})^2}{4(V_iV_j)^3},
\end{equation}
where the suffices $i,j$ run over the integers 
with the condition $1 \le i < j \le 3$.
\end{Thm}
\begin{proof}
The relations (\ref{Transform}) with ${(a, b) = (E^2, F^3)}$ 
gives the following representation: 
\begin{equation*}
\nfrac{\WC{X}}{\WC{Z}}
= (\PC{T} - \PC{S})(E^2 \PC{T} - F^3 \PC{S} ) \Big( \nfrac{X}{Z} \Big), \ \ \ 
\nfrac{\WC{Y}}{\WC{Z}}
= (\PC{T} - \PC{S})(E^2 \PC{T} - F^3 \PC{S} )^2 \Big( \nfrac{Y}{Z} \Big),
\end{equation*}
and our Weierstrass model (\ref{WModel2}) 
is transformed to the pencil of plane cubic curves 
on $\PS[2]$ with ${(X : Y : Z)}$ below:
\begin{equation}\label{CubicPencil2}
\begin{array}{l}
(E^2 \PC{T} - F^3 \PC{S} )\,Y^2Z 
 = 4(\PC{T} - \PC{S})\,X^3 -27\,\PC{T}\,XZ^2 - 27\,\PC{T}\,Z^3 \\[1.25ex]
\big( \Leftrightarrow 
\PC{T}\,( E^2 Y^2Z - 4X^3 + 27XZ^2 + 27Z^3) = \PC{S}\,( F^3 Y^2Z - 4X^3)
\ \big).
\end{array}
\end{equation}
In the pencil, 
we have the singular cubic curves (\ref{Members}) 
at the points ${\frac{\,\PC{T}\,}{\,\PC{S}\,} = \infty, 0, 1, \frac{F^3}{E^2}}$. 
These curves are represented now as follows:
\begin{equation}\label{Members2}
\left\{
\begin{array}{@{\ }l@{\quad\ }l}
E^2 Y^2Z = (X-3Z)(2X+3Z)^2
    & \text{(nodal cubic curve),} \\[1ex]
F^3 Y^2Z = 4X^3 
    & \text{(cuspidal cubic curve),} \\[1ex]
(V^4 Y^2 - XZ - Z^2)\,Z = 0 
    & \text{(conic and tangent line),} \\[1ex]
4 V^4 X^3 - F^3 XZ^2 - F^3 Z^3 = 0  
    & \text{(three concurrent lines).}
\end{array}\right.
\end{equation}
By using the polynomials $V_1, V_2, V_3$, 
we obtain the decomposition of the last singular cubic curve 
as the three concurrent lines below:
\begin{equation*}
\{ 4 (V_1 V_2)^2 X - F Z \}
\{ 4 (V_1 V_3)^2 X - F Z \}
\{ 4 (V_2 V_3)^2 X - F Z \} = 0.
\end{equation*}
Hence, together with the infinitely near base points ${(0:1:0)}$, 
we see that the remaining base points of the pencil 
are given by the following points:
\begin{equation}\label{BPt}
{}\hspace{3em}
\Big( \nfrac{X}{Z},\, \nfrac{Y}{Z} \Big)
= \Big( \nfrac{F}{4(V_i V_j)^2},\ \pm \nfrac{1}{4(V_i V_j)^3} \Big)
\hspace{2em} (\,1 \le i < j \le 3\,).
\end{equation}
The points of negative sign correspond 
to the inverse elements for the points of positive sign 
with respect to the group structure of the sections. 
By returning to the Weierstrass model, 
we obtain our representation for the generators of the sections. 
\end{proof}
%
%
\par
We mention here a consequence of the above theorem. 
As mentioned in Section \ref{Sec:2}, 
the moduli space $\MS$ of our elliptic surfaces 
is the projective line $\PS[1]$ with ${(a:b) = (E^2 : F^3)}$. 
The representation (\ref{Generator}) 
is not written rationally by the parameter of $\MS$. 
This means that some monodromy phenomenon appears in the sections. 
If we move along a closed path in $\MS$ 
within the general case 
$[\, \kt{I}_1, \kt{II}, \kt{III}, \kt{I}_0^*\,]$ 
of $\frac{\,b\,}{a} \ne \infty, 0, 1$, 
then the sections of elliptic surfaces 
are continuously transformed along the closed path. 
It induces a (generally non-trivial) permutation 
of the sections for a fixed elliptic surface 
over the point as the starting and ending of the closed path. 
This permutation is determined 
as the action of the fundamental group of 
the punctured moduli space ${\MS - \{ \infty, 0, 1 \}}$. 
\par
By the above theorem, 
this action of the fundamental group 
is described by the octahedral group 
as the covering transformations over $\MS$. 
The fundamental group is generated 
by (two of) three closed paths $\gamma_\infty, \gamma_0, \gamma_1$ 
encircling $\infty, 0, 1$, respectively once counterclockwise 
(with the composite path 
${\gamma_\infty \,{\cdot}\, \gamma_0 \,{\cdot}\, \gamma_1}$ 
homotoped to the constant path as the identity element). 
In the octahedral covering space of $\MS$, 
we have the points associated 
with the edges, the faces, the vertices of the octahedron, 
which are placed over the branched points $\infty, 0, 1$. 
Thus, the closed paths $\gamma_\infty, \gamma_0, \gamma_1$ 
determine covering transformations as rotations of angle 
$\pi, \frac{2\pi}{3}, \frac{\pi}{\,2\,}$, 
while the composite paths 
$\gamma_\infty^{\,2},\,\gamma_0^{\,3},\,\gamma_1^{\,4}$ act 
as the identity transformation, respectively. 
It determines a presentation of the octahedral group.
\par
Under the octahedral group action, 
the polynomials $V_1, V_2, V_3$ are permuted up to sign. 
They determine a three-dimensional irreducible 
linear representation of the octahedral group, 
which is isomorphic to 
the natural three-dimensional representation as rotations. 
Then the polynomials $V_1 V_2,\, V_1 V_3,\, V_2 V_3$ 
determine another representation, 
which is isomorphic to the former representation 
by taking the tensor with the sign representation. 
To sum up the above discussion, 
the representation (\ref{Generator}) 
for our generators of the sections corresponds to 
the latter three-dimensional representation of the octahedral group 
(since $F$ and $E^2$ are fixed under the group action). 
\par
We note that $E$ is not completely fixed 
by the sign representation of the octahedral group, 
and thus several correspondences between coordinates,  
such as the locally trivialized coordinates 
${(\WC{X}_1 : \WC{Y}_1 : \WC{Z}_1)}$ 
and ${(\WC{X}_2 : \WC{Y}_2 : \WC{Z}_2)}$,  
are not equivariant with respect to the octahedral group action. 
Hence, the sections (\ref{BPt}) might be observed
in the locally trivialized coordinates 
with different behavior concerning the group action. 
Such different behavior is also viewed 
for the intersection points described in the next section, 
where we use $E$ in Theorem \ref{Thm1} 
when giving complex values to the intersection points.
\section{Points in fibers intersected by sections}\label{Sec:5}%
We add remarks on the sections 
of our elliptic surfaces of the family (\ref{WModel2}). 
We will observe the intersection points 
that the sections pass through in a fixed fiber. 
\par
First, we observe them in a regular fiber. 
In working over $\C$, 
this fiber is considered to be a complex torus $\C / \Omega$ 
with a period lattice $\Omega$. 
The intersection points determines a subgroup of the complex torus. 
We denote the Weierstrass $\wp$-function by $\wp(z, \Omega)$ 
with $g_2(\Omega), g_3(\Omega)$ defined by Eisenstein series as follows: 
\begin{equation*}
\begin{array}{@{}l@{}}
\wp(z, \Omega) 
 = \nfrac{1}{z^2} + \dsum_{0 \ne \omega \in \Omega} 
    \Big\{ \nfrac{1}{(z-\omega)^2} - \nfrac{1}{\omega^2} \Big\}, \\[3ex]
g_2(\Omega) = 60 \dsum_{0 \ne \omega \in \Omega} 
                 \nfrac{1}{\omega^4}\,, \ \ 
g_3(\Omega) = 140 \dsum_{0 \ne \omega \in \Omega} 
                 \nfrac{1}{\omega^6}. \\[1ex]
\!\!\!\!\Big(\,
\wp'(z, \Omega)^2 = 4\,\wp(z, \Omega)^3 
                    - g_2(\Omega) \, \wp(z, \Omega) 
                    - g_3(\Omega).
\,\Big)
\end{array}
\end{equation*}
For our elliptic surfaces (\ref{WModel2}), 
we have $g_2(\Omega), g_3(\Omega) \ne 0$ 
since the $J$-value of the regular fiber 
is not equal to $\infty, 0, 1$. 
By Theorem \ref{MainThm}, the subgroup of $\C / \Omega$ 
is described in terms of the $\wp$-function as follows:
%
%
\begin{Thm}\label{Thm2}
For each regular fiber\/ $\C/\Omega$ 
of the rational elliptic surfaces of the family $(\ref{WModel2})$, 
the intersection points passed through by the sections 
form the subgroup generated by the points in $\C/\Omega$ 
satisfying the following condition\/$:$
\begin{equation}\label{Bpt2}
4\,\xi^3 - \nfrac{27J_0}{J_0-1}\,\xi - \nfrac{27J_0}{J_0-1} = 0
\quad 
\bigg(\, \xi = \nfrac{g_2(\Omega)\,\wp(z, \Omega)}
                 {g_3(\Omega)}, \,
       z \in \C/\Omega \bigg),
\end{equation}
where $J_0 = \frac{\,F^3}{\,E^2}$ is 
the $J$-value of the singular fiber $\kt{I}_0^*$ 
$( \kt{I}_1^*, \kt{IV}^*, \kt{III}^*$ 
in the particular cases of ${E=0,\, F=0,\, V=0}$ 
as $J_0 = \infty, 0, 1$, respectively$)$.
\end{Thm}
\begin{proof}
The value $\xi$ is determined independently 
for the scale of $\C/\Omega$, that is, 
\begin{equation*}
\nfrac{g_2(c\,\Omega)\,\wp(cz, c\,\Omega)}
      {g_3(c\,\Omega)}
 = \nfrac{g_2(\Omega)\,\wp(z, \Omega)}
       {g_3(\Omega)} \ \ \ \ (\, c \in \C^*\,),
\end{equation*}
and it is represented in terms of the coordinates 
of the elliptic surfaces $(\ref{WModel2})$ as follows:
\begin{equation*}
\nfrac{g_2(\Omega)\,\wp(z, \Omega)}
      {g_3(\Omega)}
 = \nfrac{g_2}
         {g_3} \nfrac{\mcal{X}}
                     {\mcal{Z}}
 = \nfrac{1}
         {(\PC{T} - \PC{S}) (E^2 \PC{T} - F^3 \PC{S})} 
   \nfrac{\mcal{X}}
         {\mcal{Z}}.
\end{equation*}
For the generators (\ref{Generator}) (and their inverses) of the sections, 
their corresponding values of $\xi$ 
are given by the values of $\frac{\,X\,}{Z}$ in (\ref{BPt}). 
Thus, they satisfy the condition (\ref{Bpt2}). 
\end{proof}
%
%
\par
In the general case of ${J_0 \ne \infty, 0, 1}$, 
the above cubic equation (\ref{Bpt2}) 
determines three distinct values in $\C$. 
Together with $\infty$ associated with the $0$-section, 
these correspond to the four intersection points 
among the components of the singular fiber $\kt{I}_0^*$. 
They lie on the central component, 
and its double covering space branched at them 
is isomorphic to the elliptic curve having the $J$-value $J_0$ 
with the following affine model:
\begin{equation*}
\eta^2 = 4\,\xi^3 - \nfrac{27J_0}{J_0-1}\,\xi - \nfrac{27J_0}{J_0-1}. 
\end{equation*}
Then the equation (\ref{Bpt2}) with $\eta = 0$ 
determines the $2$-division points of the elliptic curve. 
When considering $J_0$ to be a parameter, 
the above affine model gives a pencil of plane cubic curves, 
which corresponds to the rational elliptic surface of the case 
$[\, \kt{I}_1, \kt{II}, \kt{III}^* ]$.
\par
Next, we observe the intersection points 
that the sections pass through in a singular fiber, 
particularly of type $\kt{I}_1$ or $\kt{II}$. 
For our elliptic surfaces of the family (\ref{WModel2}), 
the type $\kt{I}_1$ of singular fibers appears 
except the case $[\, \kt{I}_1^*, \kt{II}, \kt{III} \,]$ of ${E=0}$, 
and the type $\kt{II}$ appears 
except the case $[\, \kt{I}_1, \kt{IV}^*, \kt{III} \,]$ of ${F=0}$. 
According to Kodaira's work \cite{Ko}, 
singular fibers have group structures 
induced by those of regular fibers around them, 
which are described for respective types of singular fibers 
as shown below. 
\begin{equation*}
\begin{array}{@{\ \ }c@{\ \ } *{4}{| @{\ \ }c@{\ \ }} }
\kt{I}_n \,(n \ge 1) 
  & \kt{I}_{2n}^*,\, \kt{I}_{2n+1}^* \,(n \ge 0) 
    & \,\kt{II},\, \kt{II}^* 
      & \,\kt{III},\, \kt{III}^* 
        & \,\kt{IV},\, \kt{IV}^* \\
\hline
\! \C^* \times \Z_n  & \C \times (\Z_2)^2,\, \C \times \Z_4 
                     & \C 
                     & \C \times \Z_2  & \C \times \Z_3  \\
\end{array}
\end{equation*}
In the above list, we denote 
by $\Z_n \,({=}\, \Z / n \Z)$ the cyclic group of order $n$, 
and by $\C$ the additive group of the complex numbers, 
and by $\C^*$ the multiplicative group of the non-zero complex numbers. 
For singular fibers of types $\kt{I}_1$ and $\kt{II}$, 
their group structures are isomorphic to $\C^*$ and $\C$, 
which are formed in the fibers by the smooth points 
except the node and the cusp, respectively. 
Then the intersection points passed through by the sections 
determine subgroups of $\C^*$ and $\C$ (up to isomorphisms of groups). 
\par
To describe these subgroups of $\C^*$ and $\C$, 
we define several polynomials 
in addition to $V_1, V_2, V_3$ as follows:
\begin{equation*}
\begin{array}{@{}l@{}}
E_1 = V_2^{\,2} - V_3^{\,2} 
    = -2\,(\alpha^4 + \beta^4 ), \\[1ex]
E_2 = V_3^{\,2} - V_1^{\,2} 
    = \alpha^4 + 6\,\alpha^2 \beta^2 + \beta^4, \\[1ex]
E_3 = V_1^{\,2} - V_2^{\,2} 
    = \alpha^4 - 6\,\alpha^2 \beta^2 + \beta^4. 
\end{array}
\ \ \ \ 
\left(
\begin{array}{@{\ }l@{\ }}
E_1 + E_2 + E_3 = 0, \\[1ex]
E_1 E_2 + E_1 E_3 + E_2 E_3 = -\frac{\,3\,}{4} F,\! \\[1ex]
E_1 E_2 E_3 = -\frac{\,1\,}{4} E.
\end{array}
\right)
\end{equation*}
We further define the following polynomials, 
firstly associated with $E_1$:
\begin{equation*}
\begin{array}{@{}l@{}}
E_1^{+} = \sqrt{2} \iu E_1 + 3 V_2 V_3
        = \phantom{-}\iu \{ (\sqrt{2}-1)^2\,\alpha^4 - (\sqrt{2}+1)^2\,\beta^4 \}, \\[1ex]
E_1^{-} = \sqrt{2} \iu E_1 - 3 V_2 V_3
        = -\iu \{ (\sqrt{2}+1)^2\,\alpha^4 - (\sqrt{2}-1)^2\,\beta^4 \},
\end{array}
\end{equation*}
and secondly associated with $E_2$:
\begin{equation*}
\begin{array}{@{}l@{}}
E_2^{+} = \sqrt{2} \iu E_2 + 3 V_1 V_3
        = \sqrt{2} \iu 
           (\alpha^2 - \sqrt{2}\,\alpha \beta + \beta^2)
           (\alpha^2 - 2 \sqrt{2}\,\alpha \beta + \beta^2), \\[1ex]
E_2^{-} = \sqrt{2} \iu E_2 - 3 V_1 V_3 
        = \sqrt{2} \iu 
           (\alpha^2 + \sqrt{2}\,\alpha \beta + \beta^2)
           (\alpha^2 + 2 \sqrt{2}\,\alpha \beta + \beta^2),
\end{array}
\end{equation*}
and thirdly associated with $E_3$:
\begin{equation*}
\begin{array}{@{}l@{}}
E_3^{+} = \sqrt{2} \iu E_3 + 3 V_1 V_2
        = \sqrt{2} \iu 
           (\alpha^2 - \sqrt{2} \iu \alpha \beta - \beta^2)
           (\alpha^2 - 2 \sqrt{2} \iu \alpha \beta - \beta^2), \\[1ex]
E_3^{-} = \sqrt{2} \iu E_3 - 3 V_1 V_2 
        = \sqrt{2} \iu 
           (\alpha^2 + \sqrt{2} \iu \alpha \beta - \beta^2)
           (\alpha^2 + 2 \sqrt{2} \iu \alpha \beta - \beta^2).
\end{array}
\end{equation*}
Each of them gives a factor of the polynomial $E$, 
whose zeros correspond to the middle points 
of certain four edges in the octahedron. 
In fact, we have the following relations among them: 
\begin{equation*}
{}\hspace{1em}
E_k^{+} E_k^{-} = E_i E_j
\ \ \ \text{and} \ \ \ 
E = -4 E_k^{+} E_k^{-} E_k^{}
\ \ \ \ \ \big( \{ i,j,k \} = \{ 1,2,3 \} \big).
\end{equation*}
The zeros of the former polynomials 
$E_1, E_2, E_3$ are determined by 
the edges not connected to either vertex of the antipodal pairs 
${(\infty, 0), (1, -1), (\iu, -\iu)}$, respectively, 
and the zeros of the latter polynomials 
$E_1^{+}, E_1^{-}, E_2^{+}, E_2^{-}, E_3^{+}, E_3^{-}$ 
are determined by the edges 
connected to $\infty, 0, 1, -1, \iu, -\iu$, respectively. 
We now describe the subgroups of $\C^*$ and $\C$ as follows:
%
%
\begin{Thm}\label{Thm1}
For the rational elliptic surfaces of the family $(\ref{WModel2})$, 
the groups of the intersection points that the sections pass through 
in the singular fibers of types $\kt{I}_1$ and\/ $\kt{II}$  
are described by the parameter ${(\alpha : \beta)}$ of the family as follows\/$:$ 
\begin{enumerate}
\item For the type\/ $\kt{I}_1\,($appearing except the case 
      $[\, \kt{I}_1^*, \kt{II}, \kt{III} \,]$ of ${E=0}\,)$, 
      the group of the intersection points is isomorphic 
      to the multiplicative subgroup of\/ $\C^*$ 
      generated by the three values of the ratios 
      $\frac{\,E_1^{+}}{\,E_1^{-}}, 
       \frac{\,E_2^{+}}{\,E_2^{-}},
       \frac{\,E_3^{+}}{\,E_3^{-}}$.
\item For the type $\kt{II} \ ($appearing except the case 
      $[\, \kt{I}_1, \kt{IV}^*, \kt{III} \,]$ of ${F=0}\,)$, 
      the group of the intersection points is isomorphic 
      to the additive subgroup of\/ $\C$ 
      generated by three representative values 
      of the homogeneous forms\/ $V_2 V_3,\, V_1 V_3,\, V_1 V_2$. 
\end{enumerate}
\end{Thm}
\begin{proof}
In the cubic pencil (\ref{CubicPencil2}) 
corresponding to the family (\ref{WModel2}), 
the singular fibers of types $\kt{I}_1, \kt{II}$ are represented by 
the nodal, cuspidal singular cubic curves of (\ref{BPt}). 
They are transformed to the following affine forms:
\begin{equation*}
E^2 \Big( \nfrac{Y}{Z} \Big)^{\!2} 
= 4 \Big( \nfrac{X}{Z} + \nfrac{3}{2} \Big)^{\!3} 
  -18 \Big( \nfrac{X}{Z} + \nfrac{3}{2} \Big)^{\!2}
\ \ \ \text{and} \ \ \ 
F^3 \Big( \nfrac{Y}{Z} \Big)^{\!2} 
= 4 \Big( \nfrac{X}{Z} \Big)^{\!3}. 
\end{equation*}
Their nodal, cuspidal points are given 
with tangent cones at them as follows:
\begin{equation*}
\begin{array}{@{}l@{}l@{}l@{}}
\kt{I}_1 : \text{node} \ 
 & \Big( \nfrac{X}{Z},\, \nfrac{Y}{Z} \Big)
   = \Big( {-}\nfrac{3}{2},\,0\,\Big) 
   & \ \ \text{with} \ \ 
     E \Big( \nfrac{Y}{Z} \Big)
     = \pm 3 \sqrt{2} \iu \Big( \nfrac{X}{Z} + \nfrac{3}{2} \Big), \\[2.5ex]
\kt{II} : \text{cusp} \ 
 & \Big( \nfrac{X}{Z},\, \nfrac{Y}{Z} \Big)
   = \Big(\, 0,\,0\,\Big) 
   & \ \ \text{with} \ \ \ 
     \Big( \nfrac{Y}{Z} \Big)^2 = 0.
\end{array}
\end{equation*}
For the group structures of the fibers, 
the identity elements are the points intersected by the $0$-section, 
which are placed upon the point ${(0:1:0)}$ 
in $\PS[2]$ with coordinates ${(X:Y:Z)}$. 
We can check that 
the groups $\C^*$ and $\C$ are isomorphic 
to the groups of the singular fibers $\kt{I}_1$ and $\kt{II}$. 
In fact, isomorphisms between them are given, respectively, 
for $\C^*$ and the fiber $\kt{I}_1$ by the following correspondence:
\begin{equation*}
\begin{array}{@{}l@{}}
\mu = \nfrac{E ( \frac{\,Y\,}{Z} )
             - 3 \sqrt{2} \iu ( \frac{\,X\,}{Z} + \frac{\,3\,}{2} )}
            {E ( \frac{\,Y\,}{Z} )
             + 3 \sqrt{2} \iu ( \frac{\,X\,}{Z} + \frac{\,3\,}{2} )} 
    \ \bigg({=}\, \nfrac{2E\,Y - 3 \sqrt{2} \iu (2X + 3Z)}
                        {2E\,Y + 3 \sqrt{2} \iu (2X + 3Z)} \bigg)\\[3ex]
\leftrightarrow\ 
\Big( \nfrac{X}{Z},\, \nfrac{Y}{Z} \Big)
= \Big( {-} \nfrac{3(\mu^2 + 10 \mu + 1)}
                  {2(\mu-1)^2},\,
        \nfrac{54\sqrt{2} \iu \mu (\mu+1)}
              {E (\mu-1)^3} \,\Big),
\end{array}
\end{equation*}
and for $\C$ and the fiber $\kt{II}$ by the following correspondence:
\begin{equation*}
\nu = \nfrac{1}{F} \nfrac{\frac{\,X\,}{Z}}
                         {\frac{\,Y\,}{Z}} 
    \ \bigg({=}\, \nfrac{1}{F} \nfrac{X}{Y} \bigg)
\ \leftrightarrow\ 
\Big( \nfrac{X}{Z},\, \nfrac{Y}{Z} \Big)
= \Big( \nfrac{F}{4 \nu^2},\,
        \nfrac{1}{4 \nu^3} \,\Big),
\end{equation*}
where ${\mu = 1}$ and ${\nu = 0}$ correspond 
to the point ${(0:1:0)}$ at infinity, 
(${\mu = 0, \infty}$ give the node and ${\nu = \infty}$ gives the cusp).
\par
Under the above isomorphisms, 
subgroups of $\C^*$ and $\C$ are determined by the intersection points 
that the sections pass through in the fibers $\kt{I}_1$ and $\kt{II}$. 
The isomorphism between $\C^*$ and the fiber $\kt{I}_1$ 
is unique up to reciprocal ratios 
(choices of the orders for the two components of the tangent cone). 
Thus, the subgroup of $\C^*$ itself 
is given without depending on the isomorphism.
On the other hand, 
the isomorphism between $\C$ and the fiber $\kt{II}$ 
is not unique by scalar multiplication of $\C^*$. 
Thus, the subgroup of $\C$ is given depending on a representative 
of the parameter $(\alpha : \beta)$ of the family (\ref{WModel2}).
\par
In the fibers $\kt{I}_1$ and $\kt{II}$, 
the generators (\ref{Generator}) of the sections 
pass through the points of (\ref{BPt}). 
Hence, the subgroups of $\C^*$ and $\C$ are generated 
by these corresponding values in $\C^*$ and $\C$. 
For the fiber $\kt{II}$, 
we immediately see that the values in $\C$ 
are represented by $V_2 V_3,\, V_1 V_3,\, V_1 V_2$. 
For the fiber $\kt{I}_1$, 
we have the following values in $\C^*$:
\begin{equation*}
{}\hspace{3em}
\mu = \nfrac{E - 3\sqrt{2} \iu (V_i V_j)\{ F + 6(V_i V_j)^2 \}}
            {E + 3\sqrt{2} \iu (V_i V_j)\{ F + 6(V_i V_j)^2 \}}
\hspace{2em} (\,1 \le i < j \le 3\,).
\end{equation*}
It is easy to check the following relations:
\begin{equation*}
{}\hspace{1em}
E = -4 E_i E_j E_k
\ \ \ \text{and} \ \ \ 
F + 6(V_i V_j)^2 = -2 E_i E_j
\ \ \ \ \ \big( \{ i,j,k \} = \{ 1,2,3 \} \big).
\end{equation*}
By using these relations, we obtain the ratios
${\frac{\,E_k^{+}}{\,E_k^{-}} \ (k=1,2,3)}$.
\end{proof}
%
%
\par
We mention several algebraic aspects of the generators 
given in the above theorem. 
Under the action of the octahedral group, 
the three ratios for the type $\kt{I}_1$ 
are permuted up to their reciprocal ratios. 
Each ratio is fixed by a certain cyclic subgroup of index $6$ 
that consists of rotations around antipodal two vertices of the octahedron. 
Thus, the three ratios determines three extension fields of degree $6$ 
over the function field of the moduli space ${\MS = \PS[1]}$ 
with the parameter ${(a : b) = (E^2 : F^3)}$, 
while the composite field of them corresponds to the function field 
of the octahedral covering space $\PS[1]$ with ${(\alpha : \beta)}$.
\par
In more detail, the three ratios are 
algebraically integral and invertible elements 
over the polynomial ring $\Z \texttt{[} J_0 \texttt{]}$ 
of ${J_0 = \frac{\,b\,}{\,a\,}}$ with integer coefficients. 
In fact, we immediately see the relations below:
\begin{equation*}
{}\hspace{1em}
E_k^{+} + E_k^{-} = 2 \sqrt{2} \iu E_k
\ \ \ \text{and} \ \ \ 
4 E_k^{\,3} -3F E_k + E = 0 
\ \ \ \ (\, k = 1,2,3 \,),
\end{equation*}
and then we have the following equalities:
\begin{equation*}
\nfrac{E_k^{+}}
      {E_k^{-}} + \nfrac{E_k^{-}\!}
                          {E_k^{+}\!} 
\,=\, - 8 \,\nfrac{E_k^{\,2}}
                  {E_k^{+} E_k^{-}} - 2
\,=\, 32 \,\nfrac{E_k^{\,3}}{E} - 2
\,=\, 24 \,\nfrac{F E_k}{E} - 10.
\end{equation*}
Thus, the three ratios and their reciprocal ratios 
satisfy the following monic polynomial equation of degree $6$ 
with coefficients in $\Z \texttt{[} J_0 \texttt{]}$:
\begin{equation*}
(\, x^2 + 10\,x + 1 \,)^3
 - 3 \cdot 12^2 J_0 \,x^2 (\, x^2 + 10\,x + 1 \,)
 + 2 \cdot 12^3 J_0 \,x^3 = 0.
\end{equation*}
Therefore, the ratios are certainly integral and invertible 
over $\Z \texttt{[} J_0 \texttt{]}$.
\par
On the other hand, the three homogeneous forms for the type $\kt{II}$ 
satisfy the following relation:
\begin{equation*}
(V_1 V_2)^2 (V_1 V_3)^2 
 + (V_1 V_2)^2 (V_2 V_3)^2 
   + (V_1 V_3)^2 (V_2 V_3)^2 = 0.
\end{equation*}
Hence, without depending on 
a representative of the parameter ${(\alpha : \beta)}$, 
the projective point ${(V_2 V_3 : V_1 V_3 : V_1 V_2)}$ in $\PS[2]$ 
is placed on the quartic curve defined by the above relation. 
This curve has three nodal points, 
which are given in association with 
${V_1=0}$, ${V_2=0}$, ${V_3=0}$, respectively. 
They correspond to the three antipodal pairs 
of the vertices for the octahedron. 
\par
Our octahedral covering space $\PS[1]$ with ${(\alpha : \beta)}$ 
is embedded into $\PS[2]$ by ${(V_1 : V_2 : V_3)}$, 
whose image is the conic defined by the following relation:
\begin{equation*}
V_1^{\,2} + V_2^{\,2} + V_3^{\,2} = 0.
\end{equation*}
Hence, the above quartic curve is considered 
to be the image of this conic by the following Cremona transformation:
\begin{equation*}
(\, V_1 : V_2 : V_3 \,) 
\ \ \mapsto \ \ 
(\, V_2 V_3 : V_1 V_3 : V_1 V_2 \,).
\end{equation*}
\par
Under the action of the octahedral group, 
the three homogeneous forms are permuted up to sign, 
and ratios among them are considered to be in the following:
\begin{equation*}
\pm \nfrac{V_2}{V_3}\,, \ \,
\pm \nfrac{V_3}{V_2}\,, \ \,
\pm \nfrac{V_1}{V_3}\,, \ \,
\pm \nfrac{V_3}{V_1}\,, \ \,
\pm \nfrac{V_1}{V_2}\,, \ \,
\pm \nfrac{V_2}{V_1}.
\end{equation*}
These $12$ ratios are algebraically determined 
over $\Z \texttt{[} J_0 \texttt{]}$ by the equation below:
%
\begin{equation*}
4(\, x^4 + x^2 + 1 \,)^3
 - J_0 (\, x^2 - 1 \,)^2 
        (\, 2x^2 + 1 \,)^2 
         (\, x^2 + 2 \,)^2 = 0.
\end{equation*}
It is shown from the following equalities 
with cyclically ordered suffices $i,j,k$ in $\{ 1, 2, 3\}$:
\begin{equation*}
\begin{array}{@{}l@{}}
J_0 = \nfrac{F^3}{E^2}
    = - \nfrac{4( V_i^{\,2} V_j^{\,2} 
                    + V_i^{\,2} V_k^{\,2} 
                         + V_j^{\,2} V_k^{\,2} )^3}
              {( V_i^{\,2} - V_j^{\,2} )^2
                   ( V_i^{\,2} - V_k^{\,2} )^2
                     ( V_j^{\,2} - V_k^{\,2} )^2} \\[2.75ex]
\phantom{J_0 = \nfrac{F^3}{E^2}}
      = \nfrac{4( V_i^{\,4} 
                   + V_i^{\,2} V_j^{\,2} 
                       + V_j^{\,4} )^3}
              {( V_i^{\,2} - V_j^{\,2} )^2
                 ( 2 V_i^{\,2} + V_j^{\,2} )^2
                   ( V_i^{\,2} + 2 V_j^{\,2} )^2}.
\end{array}
\end{equation*}
The ratio $\smash{\frac{\,V_i}{\,V_j}}$ is fixed 
by a certain rotation of angle $\pi$ around antipodal two vertices, 
and then it belongs to a extension field of degree $2$ 
over the function field determined 
by the ratio $\frac{\,E_k^{+}}{\,E_k^{-}}$. 
We can see this from the following relation:
\begin{equation*}
\hspace{3em}
\nfrac{E_k^{+}}
      {E_k^{-}}
 = \nfrac{\sqrt{2} \iu \big\{ \big( \frac{V_i}{V_j} \big)^2 -1 \big\} + 3 \big( \frac{V_i}{V_j} \big)}
         {\sqrt{2} \iu \big\{ \big( \frac{V_i}{V_j} \big)^2 -1 \big\} - 3 \big( \frac{V_i}{V_j} \big)}
\ \,\Bigg( 
{=}\  \nfrac{\sqrt{2} \iu ( V_i^{\,2} - V_j^{\,2} ) + 3 V_i V_j}
         {\sqrt{2} \iu ( V_i^{\,2} - V_j^{\,2} ) - 3 V_i V_j}
\,\Bigg).
\end{equation*}
The ratio $\frac{\,V_i}{\,V_j}$, more precisely 
${\!\iu\! \frac{\,V_i}{\,V_j} \ (\iu \!=\! \sqrt{-1}\,)}$, 
has a relationship to an object usually called the elliptic modulus, 
which we mention in the next section. 
\par
We have three values ${\xi_k\,(k=1,2,3)}$
determined by the equation (\ref{Bpt2}), 
which are given as the following ratios:
\begin{equation*}
\xi_1 = \nfrac{F}{4(V_2 V_3)^2}\,, \ \ 
\xi_2 = \nfrac{F}{4(V_1 V_3)^2}\,, \ \ 
\xi_3 = \nfrac{F}{4(V_1 V_2)^2}.
\end{equation*}
Under the action of the octahedral group, 
these ratios are permuted among them, 
each of which is fixed by a certain dihedral subgroup of index $3$ 
associated with a square in the octahedron. 
Then the ratio $\xi_k$ gives a extension field of degree $3$ 
over the function field determined by ${J_0 = \frac{\,F^3}{\,E^2}}$. 
The ratios $\frac{\,E_k^{+}}{\,E_k^{-}}$ and $\frac{\,V_i}{\,V_j}$ 
are placed over this extension field. 
In fact, we have the following equalities:
\begin{equation*}
\bigg( \nfrac{V_i^{\,2} - V_j^{\,2}}
             {V_i V_j} \bigg)^2
= \nfrac{(V_i^{\,4} + V_i^{\,2} V_j^{\,2} + V_j^{\,4}) - 3 (V_i V_j)^2}
        {(V_i V_j)^2}
= \nfrac{F}{4(V_i V_j)^2} - 3,
\end{equation*}
and thus the ratios 
$\frac{\,E_k^{+}}{\,E_k^{-}}$ and 
$\frac{\,E_k^{-}}{\,E_k^{+}}$ 
are considered to be represented as follows:
\begin{equation*}
\nfrac{E_k^{+}}
      {E_k^{-}},\,
\nfrac{E_k^{-}}
      {E_k^{+}}
\,=\, \nfrac{{\pm \sqrt{-2(\xi_k-3)}} + 3}
            {{\pm \sqrt{-2(\xi_k-3)}} - 3}
\ \,\Bigg( 
{=}\  \nfrac{2\xi_k-15 \,\mp\, 6 \sqrt{-2(\xi_k-3)}}
            {2\xi_k+3}
\,\Bigg).
\end{equation*}
Also, we have the following equalities:
\begin{equation*}
\xi_i + \xi_j + \xi_k = 0
\ \ \ \text{and} \ \ \ 
\xi_i \,\xi_j \,\xi_k 
= \nfrac{27}{4} \nfrac{J_0}{J_0-1}
= \nfrac{\xi_k^{\,3}}{\xi_k + 1},
\end{equation*}
and then $\xi_i$ and $\xi_j$ are written by $\xi_k$ as shown below:
\begin{equation*}
\xi_i,\, \xi_j
\,=\, 
\nfrac{\xi_k}{2}
\bigg( {-1} \pm \sqrt{\nfrac{\xi_k - 3}{\xi_k + 1}} \,\,\bigg).
\end{equation*}
Thus the ratios $\pm \frac{\,V_i}{\,V_j}$ and $\pm \frac{\,V_j}{\,V_i}$ 
are considered to be represented as follows:
\begin{equation*}
\pm\,\nfrac{V_i}{V_j},\,
\pm\,\nfrac{V_j}{V_i}
\ =\ 
\pm\,\sqrt{\nfrac{\xi_j}{\smash{\xi_i}}}\,,\, 
\pm\,\sqrt{\nfrac{\mathstrut \xi_i}{\smash{\xi_j}}}
\ =\ 
\pm\,\nfrac{\sqrt{\mathstrut \xi_k + 1} \,\pm \sqrt{\mathstrut \xi_k - 3}}{2}.
\end{equation*}
As shown in Section \ref{Sec:4}, 
the ambiguity arising from square roots are cleared up 
by standing on the octahedral covering space. 
\section{Characterization of octahedral covering space}\label{Sec:6}
In this section, 
we will give a characterization of our octahedral covering space 
$\PS[1]$ with the parameter ${(\alpha : \beta)}$. 
For this, we observe our octahedral polynomials 
from the viewpoint of the elliptic curve 
as the desingularization of the singular fiber $\kt{I}_0^*$. 
\par
Let $\C/\Omega_0$ be a complex torus 
that corresponds to the elliptic curve 
having the $J$-value $J_0$ of the singular fiber $\kt{I}_0^*$. 
We denote by $e_1(\Omega_0), e_2(\Omega_0), e_3(\Omega_0)$ 
the values of the $\wp$-function $\wp(z, \Omega_0)$ 
at the three $2$-division points of $\C/\Omega_0$. 
Then the following equalities are satisfied: 
\begin{equation*}
\begin{array}{@{}l@{}}
e_1(\Omega_0) + e_2(\Omega_0) + e_3(\Omega_0) = 0, \\[1ex]
e_1(\Omega_0) \, e_2(\Omega_0) 
 + e_1(\Omega_0) \, e_3(\Omega_0) 
   + e_2(\Omega_0) \, e_3(\Omega_0) 
 = - \frac{\,1\,}{4} \, g_2(\Omega_0), \\[1ex]
e_1(\Omega_0) \, e_2(\Omega_0) \, e_3(\Omega_0)
 = \frac{\,1\,}{4} \, g_3(\Omega_0).
\end{array}
\end{equation*}
Also, by the attached coefficients of our polynomials, 
we have the following equalities:
\begin{equation*}
J_0 = \nfrac{F^3}{E^2}
    = \nfrac{F^3}{F^3 - 27 (V^2)^2},
\qquad
\begin{array}{@{}l@{}}
V_1^{\,2} + V_2^{\,2} + V_3^{\,2} = 0, \\[1ex]
V_1^{\,2} V_2^{\,2} 
 + V_1^{\,2} V_3^{\,2} 
   + V_2^{\,2} V_3^{\,2} = - \frac{\,1\,}{4} F, \\[1ex]
V_1^{\,2} V_2^{\,2} V_3^{\,2} = \frac{\,1\,}{4} V^2.
\end{array}
\end{equation*}
Thus, our polynomials would be considered to be having the values 
corresponding to the following objects 
for the complex torus $\C/\Omega_0$ of the $J$-value $J_0$:
\begin{equation*}
\begin{array}{@{}l@{\ \ \qquad}l@{\ \ \qquad}l@{}}
V \leftrightarrow \sqrt{g_3(\Omega_0)},
 & V_1 \leftrightarrow \sqrt{e_1(\Omega_0)}, 
   & E_1 \leftrightarrow\, e_2(\Omega_0) - e_3(\Omega_0), \\[1ex]
E \leftrightarrow \sqrt{\Delta(\Omega_0)},
 & V_2 \leftrightarrow \sqrt{e_2(\Omega_0)},
   & E_2 \leftrightarrow\, e_3(\Omega_0) - e_1(\Omega_0), \\[1ex]
F \leftrightarrow \ \ \ g_2(\Omega_0),
 & V_3 \leftrightarrow \sqrt{e_3(\Omega_0)},
   & E_3 \leftrightarrow\, e_1(\Omega_0) - e_2(\Omega_0),
\end{array}
\end{equation*}
in which we denote the discriminant 
by ${\Delta(\Omega_0) = g_2(\Omega_0)^3 - 27 g_3(\Omega_0)^2}$.
\par
In the above correspondence, 
there are ambiguities concerning the labels of the $2$-division points 
and the square roots of their $\wp$-values. 
By choosing a fundamental basis of the lattice $\Omega_0$, 
the $2$-division points are labeled in a usual manner. 
Then the special linear group $\mathrm{SL}_2(\Z)$ permutes the labels 
through the transformations of the fundamental basis, 
while the principal congruence subgroup of level $2$ 
acts trivially on the labels. 
The complex torus $\C/\Omega_0$ is isomorphic 
to the double covering space of the singular fiber $\kt{I}_0^*$, 
and the four branch points in the central component 
are considered to be having the values 
$e_1(\Omega_0), e_2(\Omega_0), e_3(\Omega_0)$ and $\infty$. 
Then the six cross ratios among them 
are represented by our polynomials as shown below:
\begin{equation*}
\begin{array}{@{}c@{}}
-\nfrac{E_2}{E_3}\,, \ 
-\nfrac{E_3}{E_2}\,, \ 
-\nfrac{E_1}{E_3}\,, \ 
-\nfrac{E_3}{E_1}\,, \ 
-\nfrac{E_1}{E_2}\,, \ 
-\nfrac{E_2}{E_1}. \\[2.5ex]
\Big( \ \,           \lambda \,, \ 
          \nfrac{1}{\lambda} \,, \ 
                   1-\lambda \,, \ 
        \nfrac{1}{1-\lambda} \,, \  
  \nfrac{\lambda-1}{\lambda} \,, \ 
  \nfrac{\lambda}{\lambda-1}. \ \,\Big)
\end{array}
\end{equation*}
They belong to the function field 
of our intermediate covering space $\PS[1]$ 
with ${(V_1^{\,2} : V_2^{\,2})}$ 
lying over the base space ${\MS = \PS[1]}$ 
with ${(a:b) = (E^2 : F^3)}$. 
This covering space of degree $6$ is considered 
to be the (compactified) modular curve 
of the principal congruence subgroup of level $2$. 
\par
The ambiguity for the square roots seems to be more subtle. 
For our octahedral covering space $\PS[1]$ with ${(\alpha : \beta)}$, 
each of the six vertices has ramification index $4 \,({>}\,2)$ 
lying over the point of ${J_0 = 1}$ 
(not over ${J_0 = \infty}$). 
Hence, the covering space is not considered 
directly to be the modular curve 
of the principal congruence subgroup of level $4$. 
Under the fundamental basis of $\Omega_0$, 
we could not a priori choose one of the square roots 
for each $\wp$-value of the $2$-division points. 
\par
In one sense, the ambiguity for the square roots 
is equivalent to mark up the zeros of the $\wp$-function 
by $\pm z_0$ in $\C/\Omega_0$. 
In fact, for the $2$-division points, 
we have the following three functions 
(with the representation by the Weierstrass $\sigma$-functions):
\begin{equation*}
\hspace{2em}
\sqrt{\wp(z, \Omega_0) - e_k(\Omega_0)} \ 
\bigg({=}\,\nfrac{\sigma_k(z, \Omega_0)}{\sigma(z, \Omega_0)} \bigg)
\hspace{3em} (k=1,2,3),
\end{equation*}
where $\sigma(z, \Omega_0)$ is the following quasi-periodic function 
defined on the complex plane $\C$, 
and $\sigma_k(z, \Omega_0)\,(k=1,2,3)$ are 
the usual modifications of $\sigma(z, \Omega_0)$ by half-periods:
\begin{equation*}
\sigma(z, \Omega_0) 
 = z \prod_{0 \ne \omega \in \Omega_0} \!
     \Big( 1 - \nfrac{z}{\omega} \Big)
     \exp \Big( \nfrac{z}{\omega} + \nfrac{1}{2}\nfrac{z^2}{\omega^2} \Big).
\end{equation*}
Each of the above square roots gives an elliptic function 
of a certain isogenous torus of degree $2$ over $\C/\Omega_0$.  
The zero point $z_0$ in $\C/\Omega_0$ 
determines two points in the isogenous torus, 
one of which we can systematically choose 
under the fundamental basis of $\Omega_0$. 
Then we consider $\sqrt{e_k(\Omega_0)}$ multiplied by $\sqrt{-1}$ 
to be the value of the elliptic function 
at the chosen point. 
An explicit formula for the zeros of $\wp$-function
is described by Eichler and Zagier in \cite{EZ}. 
\par
From another point of view, 
the ambiguity for the square roots is explained 
in association with the complex torus $\C/\Omega$ 
determined by the following $J$-value 
and the $\wp$-values of the $2$-division points:
\begin{equation*}
J = \nfrac{J_0}{J_0 - 1},
\hspace{3em}
\begin{array}{@{}l@{}}
e_1(\Omega) = - \frac{\,1\,}{3} 
                \big(\, e_2(\Omega_0) - e_3(\Omega_0) \big), \\[1ex]
e_2(\Omega) = - \frac{\,1\,}{3} 
                \big(\, e_3(\Omega_0) - e_1(\Omega_0) \big), \\[1ex]
e_3(\Omega) = - \frac{\,1\,}{3} 
                \big(\, e_1(\Omega_0) - e_2(\Omega_0) \big).
\end{array}
\end{equation*}
By using this associated torus $\C/\Omega$, 
we have the following correspondence:
\begin{equation*}
\begin{array}{@{}l@{}}
V_1 \leftrightarrow \sqrt{e_1(\Omega_0)} 
    = \sqrt{e_2(\Omega) - e_3(\Omega)}, \\[1.25ex]
V_2 \leftrightarrow \sqrt{e_2(\Omega_0)}
    = \sqrt{e_3(\Omega) - e_1(\Omega)}, \\[1.25ex]
V_3 \leftrightarrow \sqrt{e_3(\Omega_0)} 
    = \sqrt{e_1(\Omega) - e_2(\Omega)}.
\end{array}
\end{equation*}
The ambiguity is now controlled 
by choosing a fundamental basis of the lattice $\Omega$. 
As mentioned in the above discussion of the square roots, 
we can similarly determine the square roots 
in a systematic way under the chosen basis. 
They are described in association 
with the $4$-division points of $\C/\Omega$. 
Then the principal congruence subgroup of level $4$ 
acts trivially on the square roots 
through the transformations of the fundamental basis of $\Omega$. 
\par
We now recall our octahedral covering space 
$\PS[1]$ with ${(\alpha : \beta)}$, 
which is embedded into $\PS[2]$ by ${(V_1 : V_2 : V_3)}$. 
Its defining relation of the image 
is transformed to the following form as an affine model:
\begin{equation*}
V_1^{\,2} + V_2^{\,2} + V_3^{\,2} = 0 
\ \ \Leftrightarrow \ \   
\Big( \iu \nfrac{V_1}{V_3} \Big)^2 + \Big( \iu \nfrac{V_2}{V_3} \Big)^2 = 1.
\end{equation*}
This is represented 
in terms of the associated torus $\C/\Omega$ as shown below:
\begin{equation*}
(\kappa)^2 + (\kappa')^2 = 1 
\ \ \ \bigg(\, 
\kappa  = \nfrac{\sqrt{e_3(\Omega) - e_2(\Omega)}}
                {\sqrt{e_1(\Omega) - e_2(\Omega)}}, \,\,
\kappa' = \nfrac{\sqrt{e_1(\Omega) - e_3(\Omega)}}
                {\sqrt{e_1(\Omega) - e_2(\Omega)}}
\,\bigg).
\end{equation*}
These ratios $\kappa, \kappa'$ are called 
the elliptic modulus and its complementary modulus, 
which have the following relationship 
to the parameter ${\rho = \frac{\alpha}{\,\beta\,}}$ 
of the octahedral covering space:
\begin{equation*}
(\, \kappa,\, \kappa'\,)
 = \Big( \iu \nfrac{V_1}{V_3}, \iu \nfrac{V_2}{V_3} \,\Big)
 = \Big(\, \nfrac{2\rho}{1+\rho^2},\, \nfrac{1-\rho^2}{1+\rho^2} \,\Big)
\ \ \leftrightarrow \ \ 
\rho = \nfrac{\alpha}{\beta} = \nfrac{\kappa}{1+\kappa'}.
\end{equation*}
To sum up the above discussion, by standing on the viewpoint 
of the $J$-value ${J = \frac{J_0}{\,J_0-1\,}}$, 
our octahedral covering space is characterized 
as the space of the elliptic modulus and its complementary modulus, 
in other words, 
the modular curve of the principal congruence subgroup of level $4$. 
\section*{Appendix}%
%
For benefit of readers, 
we mention the action of the (binary) octahedral group 
on the polynomials ${V_k \,(k=1,2,3)}$. 
In Section \ref{Sec:4}, 
the polynomial $V_k$ is given by a homogeneous form of degree $2$ 
defined on $\C^2$ with coordinates ${(\alpha, \beta)}$. 
The binary octahedral group acts on $\C^2$ 
by unitary transformations of determinant $1$. 
For an element $\sigma$ in the binary octahedral group, 
we denote its contragradient action on $V_k$ by $\sigma^*(V_k)$, 
that is, the composite mapping ${V_k \,\msii{\circ}\, \sigma^{-1}}$. 
By the even homogeneity of $V_k$, 
the transformation $-\sigma$ acts on it 
as the same action of $\sigma$, 
(and thus the octahedral group itself acts on $V_k$).
\begin{table}[htb]
\centering\normalsize
\caption{Action of binary octahedral group}%
\label{GroupAction}%
$
\renewcommand{\arraystretch}{1.1}
\begin{array}{@{\ }c@{\,} *{3}{| @{\ }c@{\,}} |}
\sigma & \sigma^*(V_1) & \sigma^*(V_2) & \sigma^*(V_3) \\
\hline
\hline
\pm 
\big(\! \begin{smallmatrix}
          1  &  0 \\ 
          0  &  1
    \end{smallmatrix} \!\big)  
 &  V_1  &  V_2  &  V_3  \\
\hline
\pm 
\big(\! \begin{smallmatrix}
          \iu  & \phantom{-}0 \\ 
            0  &  -\iu
    \end{smallmatrix} \!\big)  
 & V_1  & -V_2  & -V_3  \\
\pm 
\big(\! \begin{smallmatrix}
            0  & \iu \\ 
          \iu  &   0
    \end{smallmatrix} \!\big)  
 & -V_1  & V_2  & -V_3  \\
\pm 
\big(\! \begin{smallmatrix}
           0  &           -1 \\ 
           1  & \phantom{-}0
    \end{smallmatrix} \!\big)  
 & -V_1  & -V_2  & V_3  \\[.25ex]
\hline
\smash{\pm\frac{\,1\,}{2}}
\big(\! \begin{smallmatrix}
         1+\iu  &           -1+\iu \\ 
         1+\iu  & \phantom{-}1-\iu
    \end{smallmatrix} \!\big)  
 & V_2  & V_3  & V_1  \\
\smash{\pm\frac{\,1\,}{2}}
\big(\! \begin{smallmatrix}
         \phantom{-}1-\iu  & 1-\iu \\ 
                   -1-\iu  & 1+\iu
    \end{smallmatrix} \!\big)  
 & V_3  & V_1  & V_2  \\
\smash{\pm\frac{\,1\,}{2}}
\big(\! \begin{smallmatrix}
         \phantom{-}1+\iu  & 1-\iu \\ 
                   -1-\iu  & 1-\iu
    \end{smallmatrix} \!\big)  
 & -V_2  & V_3  & -V_1  \\
\smash{\pm\frac{\,1\,}{2}}
\big(\! \begin{smallmatrix}
         1-\iu  &           -1+\iu \\ 
         1+\iu  & \phantom{-}1+\iu
    \end{smallmatrix} \!\big)  
 & -V_3  & -V_1  & V_2  \\
\smash{\pm\frac{\,1\,}{2}}
\big(\! \begin{smallmatrix}
         \phantom{-}1-\iu  & 1+\iu \\ 
                   -1+\iu  & 1+\iu
    \end{smallmatrix} \!\big)  
 & -V_2  & -V_3  & V_1  \\
\smash{\pm\frac{\,1\,}{2}}
\big(\! \begin{smallmatrix}
         1+\iu  &           -1-\iu \\ 
         1-\iu  & \phantom{-}1-\iu
    \end{smallmatrix} \!\big)  
 & V_3  & -V_1  & -V_2  \\
\smash{\pm\frac{\,1\,}{2}}
\big(\! \begin{smallmatrix}
         1-\iu  &           -1-\iu \\ 
         1-\iu  & \phantom{-}1+\iu
    \end{smallmatrix} \!\big)  
 & V_2  & -V_3  & -V_1  \\
\smash{\pm\frac{\,1\,}{2}}
\big(\! \begin{smallmatrix}
         \phantom{-}1+\iu  & 1+\iu \\ 
                   -1+\iu  & 1-\iu
    \end{smallmatrix} \!\big)  
 & -V_3  &  V_1  & -V_2  \\[.25ex]
\hline
\end{array}
 \ \ \ 
\begin{array}{@{\ }c@{\,} *{3}{| @{\ }c@{\,}} |}
\sigma & \sigma^*(V_1) & \sigma^*(V_2) & \sigma^*(V_3) \\
\hline
\hline
\smash{\pm\frac{1}{\sqrt{2}}}
\big(\! \begin{smallmatrix}
      1+\iu  &     0 \\ 
          0  & 1-\iu
    \end{smallmatrix} \!\big)  
 & V_1  & V_3  & -V_2  \\
\smash{\pm\frac{1}{\sqrt{2}}}
\big(\! \begin{smallmatrix}
      1-\iu  &     0 \\ 
          0  & 1+\iu
    \end{smallmatrix} \!\big)  
 & V_1  & -V_3  & V_2  \\
\smash{\pm\frac{1}{\sqrt{2}}}
\big(\! \begin{smallmatrix}
            1  & \iu \\ 
          \iu  &   1
    \end{smallmatrix} \!\big)  
 & -V_3  &  V_2  & V_1  \\
\smash{\pm\frac{1}{\sqrt{2}}}
\big(\! \begin{smallmatrix}
         \phantom{-}1  &         -\iu \\ 
                 -\iu  & \phantom{-}1
        \end{smallmatrix} \!\big)  
 & V_3  & V_2  & -V_1  \\
\smash{\pm\frac{1}{\sqrt{2}}}
\big(\! \begin{smallmatrix}
          1  &           -1 \\ 
          1  & \phantom{-}1
    \end{smallmatrix} \!\big)  
 & V_2  & -V_1  & V_3  \\
\smash{\pm\frac{1}{\sqrt{2}}}
\big(\! \begin{smallmatrix}
         \phantom{-}1  & 1 \\ 
                   -1  & 1
    \end{smallmatrix} \!\big)  
 & -V_2  & V_1  & V_3  \\[.25ex]
\hline
\smash{\pm\frac{1}{\sqrt{2}}}
\big(\! \begin{smallmatrix}
             0  &      -1+\iu \\ 
         1+\iu  & \phantom{-}0
    \end{smallmatrix} \!\big)  
 & -V_1  & V_3  & V_2  \\
\smash{\pm\frac{1}{\sqrt{2}}}
\big(\! \begin{smallmatrix}
         \phantom{-}0  & 1+\iu \\ 
               -1+\iu  &     0
    \end{smallmatrix} \!\big)  
 & -V_1  & -V_3  & -V_2  \\
\smash{\pm\frac{1}{\sqrt{2}}}
\big(\! \begin{smallmatrix}
         \iu  &   -1 \\ 
           1  & -\iu
    \end{smallmatrix} \!\big)  
 & V_3  & -V_2  & V_1  \\
\smash{\pm\frac{1}{\sqrt{2}}}
\big(\! \begin{smallmatrix}
         \phantom{-}\iu  & \phantom{-}1 \\ 
                     -1  &         -\iu
    \end{smallmatrix} \!\big)  
 & -V_3  & -V_2  & -V_1  \\
\smash{\pm\frac{1}{\sqrt{2}}}
\big(\! \begin{smallmatrix}
         \iu  & \phantom{-}\iu \\
         \iu  &           -\iu
    \end{smallmatrix} \!\big)  
 & V_2  & V_1  & -V_3  \\
\smash{\pm\frac{1}{\sqrt{2}}}
\big(\! \begin{smallmatrix}
         \phantom{-}\iu  & -\iu \\ 
                   -\iu  & -\iu
    \end{smallmatrix} \!\big)  
 & -V_2  & -V_1  & -V_3  \\[.25ex]
\hline
\end{array}$%
\end{table}
\par
In Table \ref{GroupAction}, 
we write up the action of the binary octahedral group, 
where we denote unitary transformations of $\C^2$ 
by ${2 \times 2}$ unitary matrices in a usual sense. 
Under the group action, 
the polynomials ${\pm V_1,\, \pm V_2,\, \pm V_3}$ 
are permuted with the same behaviour of the vertices
${\frac{\alpha}{\,\beta\,} = \infty, 0,\, \pm 1,\, \pm \iu}$ 
in the embedded octahedron of the Riemann sphere 
${\PS[1] = \C \cup \{ \infty \}}$. 
For the transformations written in the left-side list of the table, 
the upper block corresponds 
to the identity transformation of the octahedron, 
and the middle block corresponds 
to the rotations of angle $\pi$ around vertices, 
and the lower block corresponds 
to the rotations of angle $\frac{2\pi}{3}$ around central points of faces. 
For the transformations written in the right-side list, 
the upper block corresponds 
to the rotations of angle $\frac{\pi}{\,2\,}$ around vertices, 
and the lower block corresponds 
to the rotations of angle $\pi$ around middle points of edges. 
\par
As mentioned in Section \ref{Sec:4}, 
the polynomial $F$ is completely fixed, 
but $V$ and $E$ may be changed to $-V$ and $-E$ 
by the action of the binary octahedral group. 
They are just fixed by the transformations 
written in the left-side list of Table \ref{GroupAction}. 
These transformations form the binary tetrahedral group 
as the normal subgroup of index $2$. 
By the upper two blocks of the left-side list, 
we have the quaternion group as the normal subgroup of index $6$. 
\par
For the action of a finite group $G$ of linear transformations, 
its Hilbert-Poincar\'e series $P(t)$ 
of the graded ring of invariant polynomials
is defined by the formal power series 
whose coefficients are the numbers of linearly independent invariants 
for homogeneous polynomials of respective degrees. 
By Molien's formula (cf.\ \cite[p.13]{Mukai}), 
it is equal to the average of reciprocals 
for reversed characteristic polynomials 
with respect to the group action, 
that is, we have the following equality:
\begin{equation*}
P(t) = \nfrac{1}{|G|} \sum_{A \in G} \nfrac{1}{\det (I-tA)},
\end{equation*}
where the group $G$ consists of linear transformations represented by matrices 
(with the identity element represented by $I\,$), 
and $|G|$ is the number of the elements in $G$. 
\par
For the action of the quaternion group, 
its Hilbert-Poincar\'e series of the invariant ring 
is equal to the following rational function: 
\begin{equation*}
\nfrac{1}{8} 
\Big\{
\nfrac{1}
      {(1-t)^2} 
 + \nfrac{1}
         {(1+t)^2} 
 + 6\,\nfrac{1}
            {1+t^2} 
\Big\}
= \nfrac{1+t^6}
        {(1-t^4)^2}
= \nfrac{1-t^{12}}
        {(1-t^4)^2 (1-t^6)}.
\end{equation*}
It means that the invariant ring is generated 
by two polynomials of degree $4$ and one polynomial of degree $6$. 
We can choose $V_1^{\,2}, V_2^{\,2}$ and $V$ as these polynomials, 
which have the following relation 
(associated with the rational double singularity of type $\RS{D}_4$):
\begin{equation*}
V^2 = 4\, V_1^{\,2} V_2^{\,2} V_3^{\,2}
    = 4\, V_1^{\,2} V_2^{\,2} (-V_1^{\,2}-V_2^{\,2}).
\end{equation*}
For the action of the binary tetrahedral group, 
we have the following rational function 
as the Hilbert-Poincar\'e series of the invariant ring:
\begin{equation*}
\begin{array}{@{}l@{}}
\nfrac{1}{24} 
\Big\{
 8\,\nfrac{1+t^6}
          {(1-t^4)^2}
 + 8\,\nfrac{1}
            {1-t+t^2} 
 + 8\,\nfrac{1}
            {1+t+t^2} 
\Big\} \\[2ex]
\ \,= \nfrac{1+t^{12}}
          {(1-t^6)(1-t^8)}
  = \nfrac{1-t^{24}}
          {(1-t^6)(1-t^8)(1-t^{12})}.
\end{array}
\end{equation*}
For the action of the binary octahedral group, 
we have the following:  
\begin{equation*}
\begin{array}{@{}l@{}}
\nfrac{1}{48} 
\Big\{
 24\,\nfrac{1+t^{12}}
           {(1-t^6) (1-t^8)} 
 + 6\,\nfrac{1}
            {1-\sqrt{2}\,t+t^2}
 + 6\,\nfrac{1}
            {1+\sqrt{2}\,t+t^2} 
 + 12\,\nfrac{1}
             {1+t^2}
\Big\} \\[2ex]
\ \,= \nfrac{1+t^{18}}
          {(1-t^8)(1-t^{12})}
  = \nfrac{1-t^{36}}
          {(1-t^8)(1-t^{12})(1-t^{18})}.
\end{array}
\end{equation*}
Thus, we can see that 
the invariant rings are generated 
by $V, F, E$ for the tetrahedral group action 
and by $F, V^2, V\! E$ for the octahedral group action 
together with the relations mentioned in Section \ref{Sec:4}.
%
%
\addtocontents{toc}{\setcounter{tocdepth}{-1}}
\section*{Acknowledgments}
\addtocontents{toc}{\setcounter{tocdepth}{2}}
I had many discussions with Professor Isao Naruki, 
and I received many pieces of helpful advice from him. 
I would like to express here my deepest gratitude to him. 
%
%
%
%
\providecommand{\bysame}{\leavevmode\hbox to3em{\hrulefill}\thinspace}
\providecommand{\MR}{\relax\ifhmode\unskip\space\fi MR }
\providecommand{\MRhref}[2]{%
  \href{http://www.ams.org/mathscinet-getitem?mr=#1}{#2}
}
\providecommand{\href}[2]{#2}

%
%
\end{document}